\documentclass[letterpaper]{amsart}

\usepackage[utf8]{inputenc}
\usepackage{lmodern}
\usepackage[T1]{fontenc}
\usepackage{amssymb}
\usepackage{enumitem}
\usepackage{mathtools}

\usepackage{tikz-cd}
\usepackage[pdfusetitle]{hyperref}
\usepackage{cleveref}

\tikzcdset{arrow style=Latin Modern}

\mathtoolsset{mathic}
\DeclareMathAlphabet\mathbfit{OML}{cmm}{b}{it}

\let\osubsection\subsection
\def\subsection#1{\osubsection{\boldmath{#1}}}

\newlist{enumarabic}{enumerate}{1}
\setlist[enumarabic]{font=\normalfont,label=(\arabic*),leftmargin=0.3in}
\newlist{enumroman}{enumerate}{1}
\setlist[enumroman]{font=\normalfont,label=(\roman*),leftmargin=0.3in}

\numberwithin{equation}{section}
\allowdisplaybreaks[4]

\theoremstyle{plain}
\newtheorem{theorem}{Theorem}[section]
\newtheorem{proposition}[theorem]{Proposition}
\newtheorem{lemma}[theorem]{Lemma}
\newtheorem{corollary}[theorem]{Corollary}

\theoremstyle{definition}

\newtheorem{remark}[theorem]{Remark}
\newtheorem{example}[theorem]{Example}

\theoremstyle{remark}
\newtheorem*{acknowledgements}{Acknowledgements}

\let\newterm\emph

\def\arxiv#1{\href{http://arxiv.org/abs/#1}{\texttt{arXiv:#1}}}
\def\doi#1{\href{http://doi.org/#1}{doi:#1}}

\def\cf{\emph{cf.}}

\let\epsilon\varepsilon
\let\phi\varphi
\let\emptyset\varnothing

\def\N{\mathbb N}
\def\Z{\mathbb Z}

\def\R{\mathbb R}
\def\C{\mathbb C}
\def\CP{\mathbb{CP}}

\def\pair#1{{\langle#1\rangle}}
\def\bigpair#1{{\bigl\langle#1\bigr\rangle}}

\def\deg#1{|#1|}

\DeclareMathOperator{\Tor}{Tor}

\DeclareMathOperator{\Hom}{Hom}

\DeclareMathOperator*{\colim}{colim}

\DeclareMathOperator{\Id}{id}

\def\kk{\Bbbk}
\def\Ll{\boldsymbol\Lambda}
\def\Sl{\mathbf S}
\def\Kl{\mathbf K}
\def\Su{\Sl^{*}}
\def\DJ{DJ}

\def\BB{\mathbf{B}}
\def\BBone{\mathbf{1}}

\def\aa{\mathbfit{a}}

\def\cc{\mathbfit{c}}
\let\topfont\mathcal
\let\algfont\mathbfit
\def\TT{\topfont{T}}

\def\KK{\topfont{K}}
\def\LL{\topfont{L}}
\def\XX{\topfont{X}}

\def\ZZ{\topfont{Z}}
\def\SS{\topfont{S}}
\def\DD{\topfont{D}}
\def\EEKK{\topfont{E K}}
\def\TTT{\algfont{T}}
\def\KKK{\algfont{K}}
\def\LLL{\algfont{L}}
\def\XXX{\algfont{X}}

\def\ZZZ{\algfont{Z}}

\def\hatphi{\tilde\phi}
\def\hatN{\tilde N}
\def\hatSigma{\tilde\Sigma}
\def\hatA{\tilde A}
\def\hata{\tilde a}

\def\DDDD{\mathbb D}
\def\SSSS{\mathbb S}

\def\SimpDJ{DJ}
\def\SimpSOne{S^{1}}

\def\PolyProd#1#2#3{(#2,#3)^{#1}}
\def\bigPolyProd#1#2#3{\bigl(#2,#3\bigr)^{#1}}

\def\AW{AW}
\def\AWhat{\skew5\widehat{\AW}}
\def\AWu#1{\AW_{\mkern -1mu #1}}
\def\AWhatu#1{\AWhat_{\mkern -2mu #1}}

\def\cupone{\mathbin{\cup_1}}

\def\SC{\mathcal{S}}

\def\timesunder#1{\mathbin{\mathchoice
  {\mathop\times\limits_{\mkern-5mu #1\mkern-5mu}}%
  {\times_{#1}}{\times_{#1}}{\times_{#1}}}}

\def\hatq{\hat{q}}
\def\tildeH{\tilde{H}}
\def\hatH{\hat{H}}

\def\smallhalf{{\textstyle{\frac{1}{2}}}}
\def\Ainfty{\(A_{\infty}\)}

\def\Ktw{\Xi}
\def\hatKtw{\hat\Ktw}
\def\hatTor#1{\mathop{\widehat{\mathrm{Tor}}}(#1)}

\def\xv#1{x_{#1}}
\def\xvi#1#2{x_{#1}^{#2}}
\def\hatawv#1#2{\hata_{#1}^{#2}}
\def\lambdavs#1#2{\lambda_{#1,#2}}
\def\tildelambdavs#1#2{\tilde\lambda_{#1,#2}}
\def\ww{o}
\def\wwonevs#1#2{\ww_{#1,#2}^{1}}
\def\wwtwovs#1#2{\ww_{#1,#2}^{2}}
\def\wwonevsv#1#2#3{\ww_{#1,#2}^{1,#3}}
\def\wwtwovsv#1#2#3{\ww_{#1,#2}^{2,#3}}
\def\tildewwonevsv#1#2#3{\tilde{\ww}_{#1,#2}^{1,#3}}
\def\tildewwtwovsv#1#2#3{\tilde{\ww}_{#1,#2}^{2,#3}}
\def\bis#1#2{b_{#1,#2}}

\def\AA{\Omega}
\def\Jtwo{I_2}
\def\Jone{I_1}

\def\ii{\mathbfit{i}}
\def\bfsigma{\boldsymbol{\sigma}}

\def\NA{N'}
\def\hatNA{\hatN'}
\def\SigmaA{\Sigma'}
\def\VA{V'}
\def\vA{v'}
\def\TA{T'}
\def\TTA{\TT'}
\def\KA{K'}
\def\LA{L'}
\def\LLA{\LL'}
\def\LLLA{\LLL'}
\def\KKA{\KK'}
\def\rhoA{(\rho')}
\def\cA{c'}
\def\RA{R'}
\def\kappaA{\kappa'}

\def\DDDDA{\DDDD'}
\def\SSSSA{\SSSS'}

\def\posetA{P'}
\def\posetB{P}
\def\parq#1#2{\ZZ_{#1}/#2}
\def\parqdef#1#2{\parq{#1}{#2}}
\def\bigparq#1#2{\ZZ_{#1}\bigm/#2}

\def\NB{N}
\def\nB{n}
\def\hatNB{\hatN}
\def\SigmaB{\Sigma}
\def\VB{V}
\def\vB{v}
\def\wB{w}
\def\TB{T}
\def\TTB{\TT}
\def\KB{K}
\def\LB{L}
\def\LLB{\LL}
\def\LLLB{\LLL}
\def\KKB{\KK}
\def\rhoB{\rho}
\def\cB{c}
\def\RB{R}
\def\kappaB{\kappa}
\def\piB{\pi}
\def\DDDDB{\DDDD}
\def\SSSSB{\SSSS}

\begin{document}

\title{Cohomology of smooth toric varieties: naturality}
\author{Matthias Franz}
\thanks{M.\,F.\ was supported by an NSERC Discovery Grant.
  X.\,F.\ was supported by the Fields Institute and by the National Research Foundation of Korea Grant NRF-2019R1A2C2010989.}
\address{Department of Mathematics, University of Western Ontario, London, Ont.\ N6A\;5B7, Canada}
\email{mfranz@uwo.ca}

\author{Xin Fu}
\address{Beijing Institute of Mathematical Sciences and Applications,  Beijing 101408, China}
\email{x.fu@bimsa.cn}

\pdfstringdefDisableCommands{\def\and{, }}
\hypersetup{pdfauthor=\authors,pdftitle={Cohomology of smooth toric varieties: naturality}}


\begin{abstract}
  Building on the recent computation of
  the cohomology rings of smooth toric varieties and partial quotients of moment-angle complexes,
  we investigate the naturality properties of the resulting isomorphism
  between the cohomology of such a space and the torsion product involving the Stanley--Reisner ring.
  If \(2\) is invertible in the chosen coefficient ring, then the isomorphism is natural
  with respect to toric morphisms, which for partial quotients are defined in analogy with toric varieties.
  In general there are deformation terms that we describe explicitly.
\end{abstract}

\maketitle

\section{Introduction}

The motivation for this paper is to understand the cohomology of smooth, possibly non-compact toric varieties.
More precisely, we want to describe the maps induced in cohomology by morphisms between such varieties.
This builds on the work~\cite{Franz:torprod} of the first author, where the cohomology rings of smooth toric varieties were described.
Recall that a toric variety~\(\XXX_{\Sigma}\) defined by a rational fan~\(\Sigma\) is a complex algebraic variety
containing an algebraic torus~\((\C^{\times})^{n}\) as a dense open subset such that the action of the torus on itself
extends to the ambient variety. Likewise, a morphism between toric varieties is a morphism of algebraic varieties
that restricts to a group morphism between the algebraic tori they contain.
As in~\cite{Franz:torprod}, our methods do not involve any algebraic geometry and work more generally
for so-called partial quotients of moment-angle complexes.
Our results have already been applied in~\cite{FuSoSong} to compute the cohomology rings of \(4\)-dimensional toric orbifolds.

The moment-angle complex~\(\ZZ_{\posetB}\) defined by a simplicial complex~\(\posetB\)
on \(m\)~vertices is a certain subspace of the product of discs~\((\DD^{2})^{m}\) stable under the canonical action of the torus~\(\TT=(\SS^{1})^{m}\).
The quotient~\(\parqdef{\posetB}{\KK}\) by a freely acting closed subgroup~\(\KK\subset\TT\) is called
a partial quotient. It comes equipped with an action of the torus~\(\LL=\TT/\KK\).
(See \Cref{sec:morph} for precise definitions.)
Topologically, smooth toric varieties are partial quotients if they are compact;
otherwise they can be equivariantly deformed into partial quotients, see \Cref{thm:toric-top}.
However, not every partial quotient arises in this way from a toric variety.
Both moment-angle complexes and partial quotients can be more generally defined for simplicial posets
instead of simplicial complexes.

We define morphisms between partial quotients that are inspired by morphisms of toric varieties
and look at the induced maps in cohomology. Since the precise definition of these morphisms
is somewhat lengthy, we state our results in the language of toric varieties in this introduction.

Let \(\kk\) be a principal ideal domain. Additively, the cohomology~\(H^{*}(\XXX_{\Sigma};\kk)\)
of a smooth (complex) toric variety~\(\XXX_{\Sigma}\) with coefficients in~\(\kk\)
was determined in the first author's doctoral dissertation~\cite{Franz:2001}, see also~\cite{Franz:2006}.
The answer is an isomorphism of graded \(\kk\)-modules
\begin{equation}
  \label{eq:iso-cohom-Tor-additive}
  H^{*}(\XXX_{\Sigma};\kk) \cong \Tor_{R}(\kk,\kk[\Sigma]).
\end{equation}
Here \(R\) is the evenly graded symmetric algebra of~\(H^1(\LL;\kk)\) or, in other words,
the evenly graded  algebra of polynomials on~\(N=H_1(\LL;\Z)\) with coefficients in~\(\kk\).
It is isomorphic to the cohomology of the classifying space of the compact torus~\(\LL\) acting on~\(\XXX_{\Sigma}\).
Moreover, \(\kk[\Sigma]\) is the Stanley--Reisner algebra of the fan~\(\Sigma\),
which is isomorphic to the \(\LL\)-equivariant cohomology of~\(\XXX_{\Sigma}\).

If \(\XXX_{\Sigma}\) is (smooth and) compact, then \(\kk[\Sigma]\) is free over~\(R\). This implies
\begin{equation}
  H^{*}(\XXX_{\Sigma};\kk) \cong 
  \kk\otimes_R\kk[\Sigma]\cong \kk[\Sigma]/\mathcal{J},
\end{equation}
where \(\mathcal{J}\) is an ideal generated by linear terms. This quotient description is known
as the Jurkiewicz--Danilov formula and in fact gives \(H^{*}(\XXX_{\Sigma};\kk)\) as an algebra,
\cf~\cite[Thm.~5.3.1]{BuchstaberPanov:2015}.\footnote{%
  It also gives the Chow ring of~\(\XXX_{\Sigma}\). At least for rational coefficients
  this isomorphism between the Chow ring and~\(\Tor_{R}^{0}(\kk,\kk[\Sigma]\)
  holds even for (smooth) non-compact~\(\XXX_{\Sigma}\), see~\cite[Sec.~3.1]{Brion:1996}.}
The multiplicative structure of~\(H^{*}(\XXX_{\Sigma};\kk)\) for non-compact~\(\XXX_{\Sigma}\)
was only computed recently in~\cite{Franz:torprod}. There is one answer
for general~\(\kk\) and a simpler one if \(2\) is invertible in~\(\kk\). Let us start with the latter case.
Recall that there is a canonical product on the torsion product in~\eqref{eq:iso-cohom-Tor-additive} because
all algebras involved are graded commutative.

\begin{theorem}[Franz]
  \label{thm:iso-mult-2}
  If \(2\) is invertible in~\(\kk\), then there is an isomorphism
  \begin{equation*}
    \tilde\Psi_{\Sigma}\colon H^{*}(\XXX_{\Sigma};\kk) \to \Tor_{R}(\kk,\kk[\Sigma])
  \end{equation*}
  that is multiplicative with respect to the canonical product on~\(\Tor\).
\end{theorem}

{
\def\LA{\LL'}
\def\LB{\LL}
Our first main result is as follows.
Recall that any morphism of smooth toric varieties~\(\phi\colon\XXX_{\SigmaA}\to\XXX_{\SigmaB}\)
induces a map of algebras~\(\phi^{*}\colon \RB\to\RA\)
and also a map
\begin{equation}
  \label{eq:map-SR-intro}
  \hatphi^{*}\colon \kk[\SigmaB] = H_{\LB}^{*}(\XXX_{\SigmaB};\kk) \to H_{\LA}^{*}(\XXX_{\SigmaA};\kk) = \kk[\SigmaA]
\end{equation}
compatible with the former.
Algebraically, the map~\eqref{eq:map-SR-intro} can be described as the pull-back of piecewise polynomials
along the map of fans~\((\NA,\SigmaA)\to(\NB,\SigmaB)\), see \Cref{rem:pullback}.

\begin{theorem}
  \label{thm:intro:nat-2}
  Suppose that \(2\) is invertible in~\(\kk\).
  The isomorphism~\(\tilde\Psi_{\Sigma}\) is natural with respect to morphisms of toric varieties.
\end{theorem}

We now turn to the case where \(2\) may not be invertible in~\(\kk\).
We consider the complex
\begin{equation}
  \label{eq:def-Kl}
  \Kl_{\Sigma} = H^{*}(\LB;\kk) \otimes \kk[\Sigma]
\end{equation}
with differential
\begin{equation}
  d = - \sum_{v\in V} \iota(\xv{v}) \otimes t_{v}
\end{equation}
where \(v\) runs through the set~\(V\) of rays in~\(\Sigma\).
Here \(t_{v}\) is the generator of~\(\kk[\Sigma]\) corresponding to~\(v\)
and \(\iota(\xv{v})\) the contraction with its minimal lattice representative~\(\xv{v}\).
The cohomology of~\(\Kl_{\Sigma}\) is the torsion product from~\eqref{eq:iso-cohom-Tor-additive}.

Choose a basis~\(x_{1}\),~\dots,~\(x_{n}\) for~\(N=H_{1}(\LB;\Z)\).
We turn \(\Kl_{\Sigma}\) into a differential graded \(\kk[\Sigma]\)-algebra by defining
\begin{equation}
  \label{eq:twisted-prod}
  \alpha * \beta = \alpha\,\beta + \sum_{i\ge j} \alpha(x_{i})\,\beta(x_{j})\,q_{ij}
\end{equation}
for~\(\alpha\),~\(\beta\in H^{1}(\LB;\kk)\subset\Kl_{\Sigma}\), where
\begin{equation}
  \label{eq:def-qij-general}
  q_{ii} = \sum_{v\in V} \frac{\xvi{v}{i}(\xvi{v}{i}-1)}{2}\,t_{v},
  \qquad
  q_{ij} = \sum_{v\in V} \xvi{v}{i}\,\xvi{v}{j}\,t_{v}
  \quad\text{if~\(i>j\).}
\end{equation}
Here \(\xvi{v}{1}\),~\dots,~\(\xvi{v}{n}\in\Z\) are the coordinates of~\(\xv{v}\in N\) with respect to the basis~\((x_{i})\),
and the fraction on the left-hand side is computed in~\(\Z\).

\begin{theorem}[Franz]
  \label{thm:iso-mult-general}
  For any principal ideal domain~\(\kk\) there is an isomorphism
  \begin{equation*}
    \Psi_{\Sigma}\colon H^{*}(\XXX_{\Sigma};\kk) \to \Tor_{R}(\kk,\kk[\Sigma])
  \end{equation*}
  that is multiplicative with respect to the twisted \(*\)-product on~\(\Tor\).
\end{theorem}

Besides the map~\eqref{eq:map-SR-intro},
a morphism of smooth toric varieties~\(\phi\colon\XXX_{\SigmaA}\to\XXX_{\SigmaB}\)
also induces a map~\(\phi^{*}\colon H^{*}(\LB;\kk)\to H^{*}(\LA;\kk)\).
We introduce the chain map
\begin{align}
  \label{eq:def-Ktw}
  \Ktw\colon \Kl_{\SigmaB} = H^{*}(\LB;\kk) \otimes \kk[\SigmaB] &\to \Kl_{\SigmaA} = H^{*}(\LA;\kk) \otimes \kk[\SigmaA], \\
  \notag
  \alpha_{i_{1}}\cdots\alpha_{i_{k}}\otimes f &\mapsto \phi^{*}(\alpha_{i_{1}})*\dots*\phi^{*}(\alpha_{i_{k}}) * \hatphi^{*}(f)
\end{align}
for~\(i_{1}<\dots<i_{k}\), where \(\alpha_{1}\),~\dots,~\(\alpha_{n}\in H^{1}(\LB;\Z)\) is the basis dual to the one for~\(H_{1}(\LB;\Z)\).
Based on~\(\Ktw\) we define the chain map
\begin{align}
  \label{eq:intro:def-hatKtw}
  \hatKtw\colon \Kl_{\SigmaB} &\to \Kl_{\SigmaA}, \\ 
  \notag
  \alpha\otimes f &\mapsto \Ktw(\alpha\otimes f)
  + \sum_{i>j} \Ktw\bigl(\iota(x_{i})\,\iota(x_{j})\,\alpha \otimes f\bigr)\,\hatq_{ij}
\end{align}
where the elements~\(\hatq_{ij}\in\Z[\SigmaA]\) can be written explicitly in terms of the map of fans~\((\NA,\SigmaA)\to(\NB,\SigmaB)\), see \Cref{thm:hatq}.
We write the map induced by~\(\hatKtw\) in cohomology as~\(\hatTor{\phi}\).
}

\begin{theorem}
  \label{thm:intro:nat}
  Let \(\phi\colon\XXX_{\SigmaA}\to\XXX_{\SigmaB}\) be a morphism of smooth toric varieties.
  The diagram
  \begin{equation*}
    \begin{tikzcd}
      H^{*}(\XXX_{\SigmaB};\kk) \arrow{d}[left]{\Psi_{\SigmaB}}[right]{\cong} \arrow{r}{\phi^{*}} & H^{*}(\XXX_{\SigmaA};\kk) \arrow{d}[right]{\Psi_{\SigmaA}}[left]{\cong} \\
      \Tor_{\RB}(\kk,\kk[\SigmaB]) \arrow{r}{\hatTor{\phi}} & \Tor_{\RA}(\kk,\kk[\SigmaA])
    \end{tikzcd}
  \end{equation*}
  commutes for any principal ideal domain~\(\kk\).
\end{theorem}

The paper is organized as follows: In \Cref{sec:alg}, we review some differential-algebraic background
and prove several auxiliary results related to bar constructions.
Motivated by morphisms of toric varieties, we introduce in \Cref{sec:morph}
the class of morphisms that we are focusing on,
namely toric morphisms of simplicial posets.
They induce equivariant maps between the associated partial quotients.
In \Cref{sec:BT,sec:DJ} we revisit and extend the formality result for classifying spaces of tori
and for Davis--Janusz\-kie\-wicz spaces that was established in~\cite{Franz:gersten} and
laid the ground for the results from~\cite{Franz:torprod} recalled above.
These formality maps depend on the choice of certain representatives. We show that the maps
for different choices are related via algebra homotopies,
and we investigate how these homotopies interact with \(\cupone\)-products.
In \Cref{sec:main} we recall the setup for the proofs of \Cref{thm:iso-mult-2,thm:iso-mult-general} and
prove our technical main result, \Cref{thm:main}.
We deduce \Cref{thm:intro:nat-2,thm:intro:nat} as special cases of it in \Cref{sec:nat-general,sec:nat-2}.
As further applications we study the map induced in cohomology by
the projection~\(\ZZZ_{\Sigma}\to\XXX_{\Sigma}\) (\Cref{sec:macs}), and
we describe an automorphism of~\(\Kl_{\Sigma}\) that in cohomology intertwines the twisted with
the canonical product (\Cref{sec:compare}).
We illustrate our results with numerous examples.

\begin{acknowledgements}
  We thank Taras Panov for a comment that inspired \Cref{rem:I-Sigma-large}.
\end{acknowledgements}

\section{Algebras, coalgebras and bar constructions}
\label{sec:alg}

Throughout this paper, we work over a principal ideal domain~\(\kk\).
We refer to~\cite[Sec.~2]{Franz:gersten} for our sign conventions as well as
for the definitions of coaugmented differential graded coalgebras (dgcs),
augmented differential graded algebras (dgas) and bar constructions of the latter.
We write~\(\BBone=\iota(1)\in\BB A\) for the unit of the bar construction of an augmented dga~\(A\)
with the canonical coaugmentation~\(\iota\colon\kk\to\BB A\) and \(t_{A}\colon \BB A\to A\)
for the canonical twisting cochain.

We recall from~\cite[\S 1.11]{Munkholm:1974} the definitions of algebra and coalgebra homotopies.
Let \(f\),~\(g\colon A\to B\) be morphisms of augmented dgas.
An \newterm{algebra homotopy} from~\(f\) to~\(g\) is a map~\(h\colon A\to B\) such that
\begin{equation}
  \label{eq:def-alg-h}
  d(h) = g-f,
  \qquad
  h\,\mu_{A} = \mu_{B}\,(f\otimes h + h\otimes g)
  \qquad\text{and}\qquad
  \epsilon_{B}\,h = 0,
\end{equation}
where \(\mu_{A}\) and~\(\mu_{B}\) denote the multiplication maps of~\(A\) and~\(B\), respectively.\footnote{%
  Note that \(-h\) is an algebra homotopy in the sense of~\cite{Munkholm:1974}.
  The conditions on the (co)aug\-men\-ta\-tions in~\eqref{eq:def-alg-h} and~\eqref{eq:def-coalg-h} are missing in~\cite{Munkholm:1974}.
  Without them the bijections with twisting cochain homotopies that are stated in~\cite{Munkholm:1974} (and which we will not use) may fail.}
Similarly, let \(f\),~\(g\colon C\to D\) be morphisms of coaugmented dgcs.
A \newterm{coalgebra homotopy} from~\(f\) to~\(g\) is a map~\(h\colon C\to D\) such that
\begin{equation}
  \label{eq:def-coalg-h}
  d(h) = g-f,
  \qquad
  \Delta_{D}\,h = (f\otimes h + h\otimes g)\,\Delta_{C}
  \qquad\text{and}\qquad
  h\,\iota_{C} = 0.
\end{equation}

Remember that the dual of a coaugmented dgc is an augmented dga.
Moreover, if \(f\colon C\to D\) is a morphism of augmented dgcs,
then its transpose~\(f^{*}\colon D^{*}\to C^{*}\) is a morphism of augmented dgas.
(See~\cite[eq.~(2.4)]{Franz:gersten} for our sign convention regarding the transpose.)

\begin{lemma}
  \label{thm:dual-coalg-h}
  Let \(f\),~\(g\colon C\to D\) be morphisms of coaugmented dgcs.
  If \(h\colon C\to D\) is a coalgebra homotopy from~\(f\) to~\(g\),
  then its transpose~
  \( 
    h^{*}\colon D^{*}\to C^{*}
  \) 
  is an algebra homotopy from~\(f^{*}\) to~\(g^{*}\).
\end{lemma}

\begin{proof}
  This follows directly from the properties of the transpose as given in~\cite[eqs.~(2.5),~(2.7)]{Franz:gersten}.
\end{proof}

We consider the following diagram of augmented dgas.
\begin{equation}
  \label{eq:diag-A-B1-B2}
  \begin{tikzcd}[column sep=large]
    A \arrow{d} \arrow{r}{g} & A' \arrow{d} \\
    B_{1} \arrow{d}[left]{f} \arrow{r}{g_{1}} \arrow[dashed]{rd}{h} & B_{1}' \arrow{d}{f'} \\
    B_{2} \arrow{r}[below]{g_{2}} & B_{2}'
  \end{tikzcd}
\end{equation}
We assume that the top square commutes on the nose
and the bottom square up to the algebra homotopy~\(h\colon B_{1}\to B_{2}'\)
from~\(g_{2}\,f\) to~\(f'\,g_{1}\). We consider the map
\begin{align}
  \label{eq:def-Theta-h}
  \Theta_{h}\colon \BB(\kk,A,B_{2}) &\to \BB(\kk,A',B_{2}'), \\
  \notag [a_{1}|\dots|a_{k}]\otimes b &\mapsto
  \begin{cases}
    \BBone \otimes g_{2}(b) & \text{if \(k=0\),} \\
    \bigl[g(a_{1})|\dots|g(a_{k})\bigr]\otimes g_{2}(b) \\
    \quad + \bigl[g(a_{1})|\dots|g(a_{k-1})\bigr]\otimes h(a_{k})\,g_{2}(b) & \text{if \(k>0\)}
  \end{cases}
\end{align}
between one-sided bar constructions.
Using the isomorphisms of graded modules~\(\BB(\kk,A,B_{2})=\BB A\otimes B_{2}\)
and~\(\BB(\kk,A',B_{2}')=\BB A'\otimes B_{2}'\), the map~\(\Theta_{h}\) can also be written as
\begin{equation}
  \label{eq:def-Theta-h-bis}
  \Theta_{h} = \BB g \otimes g_{2} + (1\otimes \mu_{B_{2}'})(1\otimes h\,t_{A}\otimes g_{2})(\Delta_{\BB A} \otimes 1)
\end{equation}
where \(t_{A}\colon\BB A\to A\) is the canonical twisting cochain.
Note that we are not indicating the map~\(A\to B_{1}\) when writing \(h\,t_{A}\) above.

\begin{lemma}
  \label{thm:homotopy-bar}
  The map~\(\Theta_{h}\) is a chain map. Moreover, the diagram of complexes
  \begin{equation*}
    \begin{tikzcd}[column sep=huge]
      \BB(\kk,A,B_{1}) \arrow{d}[left]{\BB(1,1,f)} \arrow{r}{\BB(1,g,g_{1})} & \BB(\kk,A',B_{1}') \arrow{d}{\BB(1,1,f')} \\
      \BB(\kk,A,B_{2}) \arrow{r}{\Theta_{h}} & \BB(\kk,A',B_{2}')
    \end{tikzcd}
  \end{equation*}
  commutes up to homotopy.
\end{lemma}

\begin{proof}
  One directly verifies the first claim and also that
  \begin{align}
    H=\BB g\otimes h\colon \BB(\kk,A,B_{1}) &\to \BB(\kk,A',B_{2}'), \\
    \notag {\underbrace{[a_{1}|\dots|a_{k}]}_{\aa}}\otimes b &\mapsto
    (-1)^{\deg{\aa}}\,\bigl[g(a_{1})|\dots|g(a_{k})\bigr]\otimes h(b)
  \end{align}
  is a homotopy from~\(\BB(1,1,f')\,\BB(1,g,g_{1})\) to~\(\Theta_{h}\,\BB(1,1,f)\).
\end{proof}

Recall that a homotopy Gerstenhaber algebra (hga) is a dga~\(A\) equipped with operations
\begin{equation}
  E_{k}\colon A\otimes A^{k} \to A
\end{equation}
for~\(k\ge 1\) that allow to define a multiplication on~\(\BB A\) turning the bar construction into a dg~bialgebra,
compare~\cite[Sec.~4]{Franz:gersten} or~\cite[Sec.~2]{Franz:torprod}.
We write the product of two elements~\(\aa\),~\(\aa'\in~\BB A\) as~\(\aa\circ\aa'\).
Cochain algebras of simplicial sets and topological spaces are naturally hgas,
and any (graded) commutative dga is an hga with~\(E_{k}=0\) for all~\(k\ge1\).
Also recall that
\begin{equation}
  a\cupone b=-E_{1}(a;b)
\end{equation}
is a \(\cupone\)-product for~\(A\) satisfying the Hirsch formula.
This means that for all~\(a\),~\(b\),~\(c\in A\) one has
\begin{gather}
  d(a\cupone b)+ da\cupone b+(-1)^{\deg{a}}\,a\cupone db = ab - (-1)^{\deg{a}\deg{b}}\,ba, \\
  ab \cupone c = (-1)^{\deg{a}}\,a(b\cupone c) + (-1)^{\deg{b}\deg{c}}(a\cupone c)\,b.
\end{gather}

We will need the following fact.

\begin{lemma}
  \label{thm:iterated-cup1}
  For any~\(k\ge1\) and~\(a_{1}\),~\dots,~\(a_{k}\in A\) of even positive degrees one has
  \begin{equation*}
    t_{A}\bigl([a_{1}]\circ\dots\circ[a_{k}]\bigr) =
    (-1)^{k-1}\,\bigl(((a_{1}\cupone a_{2})\cupone a_{3})\cupone\cdots\bigr)\cupone a_{k}.
  \end{equation*}
\end{lemma}

As before, \(t_{A}\colon\BB A\to A\) denotes the canonical twisting cochain.

\begin{proof}
  For all~\(a\),~\(b\in A\) of positive degrees we have
  \begin{equation}
    t_{A}\bigl([a]\circ[b]\bigr) = 
    (-1)^{\deg{a}}\,E_{1}(a;b) = (-1)^{\deg{a}+1}\,a\cupone b,
  \end{equation}
  compare~\cite[eqs.~(6.7)~\&~(6.10)]{Franz:homog}.
  Our claim follows inductively from this by evaluating the iterated product from the left to the right.
\end{proof}

Given a morphism~\(A\to B\) of hgas,
the one-sided bar construction~\(\BB(\kk,A,B)\) can be turned into a dga, compare~\cite[Prop.~2.2]{Franz:torprod}.
If all operations~\(E_{k}\) with~\(k\ge1\) are trivial in~\(B\), then the product on~\(\BB(\kk,A,B)\) is componentwise, that is,
\begin{equation}
  \label{eq:prod-BB-onesided-comm}
  (\aa\otimes b)\circ (\aa'\otimes b') = (-1)^{\deg{\aa'}\deg{b}}\,\aa\circ\aa' \otimes b\,b'
\end{equation}
for all~\(\aa\),~\(\aa'\in\BB A\) and~\(b\),~\(b'\in B\).

\section{Morphisms}
\label{sec:morph}

As in~\cite{Franz:torprod}, we write algebraic varieties in the form~\(\XXX\),
topological spaces as~\(\XX\) and simplicial sets as~\(X\).
As discussed in~\cite[Sec.~3.3]{Franz:torprod},
any lattice~\(N\cong\Z^{n}\) gives naturally rise to a simplicial torus~\(T_{N}=BN\), a compact torus~\(\TT_{N}\cong(\SS^{1})^{n}\)
and an algebraic torus~\(\TTT_{N}\cong(\C^{\times})^{n}\).
Here \(\SS^{1}\subset\C\) denotes the unit circle.

To motivate our definition of a ``toric morphism'', we start with the case of toric varieties.
Recall that morphisms of toric varieties~\(\phi\colon\XXX_{\SigmaA}\to\XXX_{\SigmaB}\)
are morphisms of complex algebraic tori~\(\LLLA=\TTT_{\NA}\to\LLLB=\TTT_{\NB}\)
that algebraically (or just continuously in the metric topology) extend to the toric varieties containing them.
They correspond to morphisms of fans~\(A\colon(\NA,\SigmaA)\to(\NB,\SigmaB)\), that is,
to morphisms of lattices~\(A\colon \NA\to \NB\) such that
\begin{equation}
  \label{eq:toric-morph}
  \forall\,\sigma\in\SigmaA\;\;\;\exists\,\tau\in\SigmaB\quad A\,\sigma\subset\tau,
\end{equation}
compare~\cite[Sec.~1.3]{Fulton:1993}.

Also recall from~\cite[p.~78]{Fulton:1993} and~\cite[Sec.~2]{Franz:2010}
that toric varieties can be defined over any topological monoid~\(\DDDD\)
with group of invertible elements~\(\SSSS\subset\DDDD\).
We write such a toric variety as~\(\XXX_{\Sigma}(\DDDD)\), so that \(\XXX_{\Sigma}=\XXX_{\Sigma}(\C)\).

For a given rational fan~\(\Sigma\) in the vector space~\(N_{\R}=N\otimes_{\Z}\R\),
we write \(x_{v}\in N\) for the the minimal lattice representative of the ray~\(v\in V\).
If these vectors do not span the lattice~\(N\), we add ``ghost rays'' to~\(V\)
such that the minimal lattice representatives of both the real and the ghost rays
do span \(N\). Note that the ghost rays do not appear in~\(\Sigma\).

Let \(\hatN=\Z^{V}\), and let \(\hatSigma\) be the fan in~\(\hatN_{\R}\)
that is combinatorially equivalent to~\(\Sigma\) and whose rays
are the canonical basis elements~\(e_{v}\in\hatN\) with~\(v\in V\). Also,
let \(\ZZZ_{\Sigma}(\DDDD)=\XXX_{\hatSigma}(\DDDD)\) be the Cox construction of~\(\XXX_{\Sigma}(\DDDD)\)
and write \(\TTT=\TTT_{\hatN}\). Then
\begin{equation}
  \XXX_{\Sigma} = \ZZZ_{\Sigma}/\KKK
\end{equation}
for some closed subgroup~\(\KKK\subset\TTT\) acting freely on~\(\ZZZ_{\Sigma}\).

Let \(\DD^{2}\subset\C\) be the unit disc, considered as a monoid with group of invertible elements \(\SS^{1}\).
Note that \(\ZZ_{\Sigma}=\ZZZ_{\Sigma}(\DD^{2})\) is nothing but the moment-angle complex
\begin{equation}
  \ZZ_{\Sigma} = \bigl(\DD^{2},\SS^{1}\bigr)^{\Sigma} = \colim_{\sigma\in\Sigma} \bigl(\DD^{2},\SS^{1}\bigr)^{\sigma}
\end{equation}
defined by~\(\Sigma\). Here we are using the simplicial version of the polyhedral product functor,
compare~\cite[Secs.~4.2~\&~8.1]{BuchstaberPanov:2015}.
We also write \(\XX_{\Sigma}=\XXX_{\Sigma}(\DD^{2})\)
as well as \(\LL\), \(\TT\) and~\(\KK\) for the compact abelian groups
corresponding to~\(\LLL\),~\(\TTT\) and~\(\KKK\).

\begin{lemma}
  \label{thm:toric-top}
  Let \(\Sigma\) be a rational fan.
  \begin{enumroman}
  \item \label{thm:toric-X-Z-K}
    The projection~\(\ZZ_{\Sigma}\to\XX_{\Sigma}\) induces a homeomorphism
    \begin{equation*}
      \XX_{\Sigma} = \ZZ_{\Sigma}/\KK.
    \end{equation*}
  \item \label{thm:toric-retract}
    The inclusion
    \begin{equation*}
      \XX_{\Sigma} \hookrightarrow \XXX_{\Sigma}
    \end{equation*}
    is an \(\LL\)-equivariant strong deformation retract.
  \end{enumroman}
\end{lemma}

\begin{proof}
  The first part is implicit in the proof of~\cite[Prop.~4.1]{Franz:torprod}.
  This latter result is the second part above.
\end{proof}

We now generalize from smooth toric varieties to partial quotients
as in~\cite[Sec.~4]{Franz:torprod}, see also~\cite[Sec.~2.8]{BuchstaberPanov:2015}. 

Let \(\posetB\) be a simplicial poset on the vertex set~\(V\).
The empty simplex~\(\emptyset\) is always contained in~\(\posetB\). 
We furthermore
assume \(\posetB\) and~\(V\) to be finite and \(V\) additionally ordered
and possibly containing ghost vertices not appearing in~\(\posetB\).
By abuse of notation, we denote the complete simplicial complex on~\(V\) by the same letter.
We will also use of the folding map~\(\posetB\to V\), which sends each simplex~\(\sigma\in \posetB\) to its vertex set~\(V(\sigma)\subset V\).

Given a simplicial poset~\(\posetB\), we define the moment-angle complex
\begin{equation}
    \ZZ_{\posetB} = \colim_{\sigma\in P}\ZZ_{\sigma}
\end{equation}
where
\begin{equation}
    \ZZ_{\sigma} = (\DD^{2},\SS^{1})^{\sigma} \coloneqq (\DD^{2})^{\sigma} \times (\SS^{1})^{V\setminus\sigma}.
\end{equation}
The torus~\(\TT=(\SS^{1})^{V}\) canonically acts on~\(\ZZ_{\posetB}\).

Let \(\KK\subset\TT=(\SS^{1})^{V}\) be a closed subgroup acting freely on~\(\ZZ_{\posetB}\).
Based on~\(\posetB\) and~\(\KK\) we define the partial quotient~\(\parqdef{\posetB}{\KK}\).
It comes equipped with a canonical action of the quotient group~\(\LL=\TT/\KK\); the quotient map
\begin{equation}
  \kappa\colon \ZZ_{\posetB} \to \parqdef{\posetB}{\KK}
\end{equation}
is equivariant with respect to the projection~\(\TT\to\LL\).
We also write \(\kappa\) for the projection map~\(\hatN=\Z^{V}=H_{1}(\TT;\Z)\to N=H_{1}(\LL;\Z)\)
between lattices as well as for the one between the corresponding real vector spaces.

For any~\(\sigma\in\posetB\),
the restriction~\(\kappa\colon\ZZ_{\sigma}\to\parqdef{\sigma}{\KK}\)
is the Cox construction
discussed above in the special case~\(\DDDD=\DD^{2}\). If we consider \(\sigma\) as a cone in~\(\hatN_{\R}\)
spanned by the~\(e_{v}\) with~\(v\in V\), then \(\kappa(\sigma)\) is a cone in~\(\N_{\R}\)
whose rays~\(\xv{v}=\kappa(e_{v})\in N\) extend to a basis for~\(N\).

\begin{remark}
  To every smooth toric variety~\(\XX_{\Sigma}\) one can associate a partial quotient as in \Cref{thm:toric-top}.
  Then \(\kappa(\sigma)=\sigma\) for each cone~\(\sigma\in\Sigma\).
  In general, however, the cones~\(\kappa(\sigma)\) defined above do not form a fan.
  The fan property is crucial in algebraic geometry when gluing together the affine toric varieties
  associated to each cone. The topological setting therefore is strictly more general than that of smooth toric varieties,
  compare~\cite[Ex.~1.19]{DJ:1991}.
\end{remark}

Consider another partial quotient~\(\parqdef{\posetA}{\KK'}\).
Let \(A\colon \NA\to \NB\) be a morphism of lattices,
and let \(\nu\colon\posetA\to\posetB\) be an order-preserving function such that
\begin{equation}
  A\,\kappaA(\sigma) \subset \kappaB(\nu(\sigma))
\end{equation}
for all~\(\sigma\in\posetA\).
We call the pair~\((A,\nu)\) a \newterm{toric morphism of simplicial posets} and write it as
\begin{equation}
  (A,\nu)\colon (\NA,\posetA) \to (\NB,\posetB).
\end{equation}
It induces a morphism~%
\(
  \parqdef{\sigma}{\KK'} \to \parqdef{\nu(\sigma)}{\KK}
\)
for each~\(\sigma\in\posetA\).
Given that \(\nu\) is order-preserving, these morphisms glue together to a continuous map
\begin{equation}
  \phi=\phi_{(A,\nu)}\colon \parq{\posetA}{\KK'} = \colim_{\sigma\in\posetA} \parq{\sigma}{\KK'} \to \colim_{\tau\in\posetB} \parq{\tau}{\KK} = \parq{\posetB}{\KK},
\end{equation}
equivariant with respect to the morphism of groups~\(\LLA\to\LLB\) determined by~\(A\).

We also get a map between the Cox constructions:
The morphism~\(A\colon \NA\to \NB\) lifts to a morphism of lattices
\( 
  \hatA\colon \hatNA \to \hatNB
\) 
such that
\begin{equation}
  \label{eq:pi-A-e-x}
  \kappaB(\hatA\,e_{\vA}) = A\,\kappaA(e_{\vA}) = A\,\xv{\vA}
  \qquad\text{and}\qquad
  \hatA\,\sigma \subset \nu(\sigma)\subset\hatNB
\end{equation}
for all~\(\vA\in\VA\) and all~\(\sigma\in\posetA\).
In particular, we have a toric morphism of simplicial posets
\begin{equation}\label{eq:hatA and nu}(\hatA,\nu)\colon(\hatNA,\posetA)\to(\hatNB,\posetB).
\end{equation}
It is canonically represented by a matrix~\((\hatawv{\vB}{\vA})\in\Z^{\VB\times\VA}\).
For a non-ghost vertex~\(\vA\in\VA\) the coefficients~\(\hatawv{\vB}{\vA}\) are in fact non-negative
because \(\hatA\,e_{\vA}\) is an \(\N\)-linear combination of the basis elements~\(e_{\vB}\)
where \(\vB\) runs through the vertices of the simplex~\(\nu(\vA)\in\posetB\).

\begin{remark}
  \label{rem:hatA-unique}
  If \(\VA\) does not contain ghost vertices, then the conditions~\eqref{eq:pi-A-e-x} determine \(\hatA\) uniquely.
  In general, the coefficients~\(\hatawv{\vB}{\vA}\) for the ghost vertices~\(\vA\in\VA\) will depend on the chosen lift of~\(A\).
  We will see in \Cref{rem:hatqij-unique} that this ambiguity disappears when we construct the twisting term~\(\hatq_{ij}\)
  that via the definition~\eqref{eq:intro:def-hatKtw} enter the map~\(\hatTor{\phi}\) appearing in \Cref{thm:intro:nat}.
\end{remark}

From \(\hatA\) and~\(\nu\) we obtain the map
\begin{equation}
  \hatphi=\phi_{(\hatA,\nu)}\colon \ZZ_{\posetA} \to \ZZ_{\posetB},
\end{equation}
which is equivariant with respect to the map~\(\TTA\to\TTB\) given by~\(\hatA\).
It follows from~\eqref{eq:pi-A-e-x} and \Cref{thm:toric-top}\,\ref{thm:toric-X-Z-K} that \(\hatphi\) fits into the commutative diagram
\begin{equation}
  \begin{tikzcd}
    \ZZ_{\posetA} \arrow{d}[left]{\kappaA} \arrow{r}{\hatphi} & \ZZ_{\posetB} \arrow{d}{\kappaB} \\
    \parq{\posetA}{\KK'} \arrow{r}{\phi} & \parq{\posetB}{\KK} \mathrlap{.}
  \end{tikzcd}
\end{equation}

For~\(\sigma\in\posetA\), the restriction of~\(\hatphi\) to the ``affine charts''
\begin{equation}
  (\DD^{2},\SS^{1})^{\sigma} = \ZZ_{\sigma} \stackrel{\hatphi}{\longrightarrow} \ZZ_{\nu(\sigma)} = (\DD^{2},\SS^{1})^{\nu(\sigma)}
\end{equation}
is given in the canonical coordinates indexed by~\(\VA\) and~\(\VB\) as
\begin{equation}
  \label{eq:formula-hatphi}
  \hatphi(x)_{\vB} = \prod_{\vA\in\VA} \xv{\vA}^{\hatawv{\vB}{\vA}}.
\end{equation}
We point out that our toric morphisms of simplicial posets
lead to maps~\(\hatphi\) that are more general than the simplicial maps
considered for example in~\cite[Prop.~4.2.4]{BuchstaberPanov:2015}.

It will be useful to take an even more general point of view.
Let \(\DDDD\) be any commutative simplicial monoid and~\(\SSSS\subset\DDDD\) a simplicial subgroup.
For example, we may consider the pair~\((\DDDD,\SSSS)=(BS^{1},b_{0})\),
where \(BS^{1}\) is the simplicial classifying space of the simplicial circle~\(S^{1}=T_{\Z}\) and \(b_{0}\) its unique base point.
(Here we are using that the simplicial classifying space of a commutative simplicial group is again a commutative simplicial group.
More generally, the simplicial Borel construction of a commutative simplicial group (or monoid)
with respect to another commutative simplicial group is again a commutative simplicial group (or monoid),
compare~\cite[p.~98]{May:1968}.)
The polyhedral product associated to the pair~\((BS^{1},b_{0})\) is the simplicial Davis--Januszkiewicz space
\begin{equation}
  \label{eq:def-DJ}
  \DJ_{\posetB} = (BS^{1},b_{0})^{\posetB}.
\end{equation}

Using \eqref{eq:formula-hatphi} as a definition, we get from~\((\hatA,\nu)\)
a map of simplicial monoids
\begin{equation}
  \label{eq:def-phi-A-nu-sigma}
  (\DDDD,\SSSS)^{\sigma} \to (\DDDD,\SSSS)^{\nu(\sigma)}
\end{equation}
for any pair~\((\DDDD,\SSSS)\) as above and any~\(\sigma\in\posetA\).

\begin{lemma} \( \)
  \label{thm:pp-toric-natural}
  \begin{enumroman}
  \item
    The maps~\eqref{eq:def-phi-A-nu-sigma} induce a well-defined morphism of simplicial monoids
    \begin{equation*}
      \hatphi \colon (\DDDD,\SSSS)^{\posetA} \to (\DDDD,\SSSS)^{\posetB}.
    \end{equation*}
  \item
    This way, the polyhedral product~\((\DDDD,\SSSS)^{\posetB}\) becomes bifunctorial
    with respect to multiplicative maps of pairs~\((\DDDDA,\SSSSA)\to(\DDDDB,\SSSSB)\)
    and toric morphisms of simplicial posets~\((\hatA,\nu)\colon(\hatNA,\posetA)\to(\hatNB,\posetB)\).
  \end{enumroman}
\end{lemma}

\begin{proof}
  It is readily verified that the diagram
  \begin{equation}
    \begin{tikzcd}
      (\DDDDA,\SSSSA)^{\sigma} \arrow{d} \arrow{r} & (\DDDDB,\SSSSB)^{\sigma} \arrow{d} \\
      (\DDDDA,\SSSSA)^{\nu(\sigma)} \arrow{r} & (\DDDDB,\SSSSB)^{\nu(\sigma)}
    \end{tikzcd}
  \end{equation}
  commutes for every~\(\sigma\in\posetA\), which proves both claims.
\end{proof}

Because of the bifunctoriality we get an action of the group~\(\SSSS^{V} = (\SSSS,\SSSS)^{\posetB}\) on~\((\DDDD,\SSSS)^{\posetB}\)
such that the map~\(\hatphi\) is equivariant with the induced morphism of groups~\(\SSSS^{\VA}\to\SSSS^{\VB}\).

We write \(Y_{G}=EG\timesunder{G}Y\) for the simplicial Borel construction of the \(G\)-space~\(Y\) with respect to the simplicial group~\(G\),
compare~\cite[Sec.~3.2]{Franz:torprod}.
In the proof of~\cite[Lemma~5.1]{Franz:torprod} we observed that both arrows in the zigzag
\begin{equation}
  \label{eq:Borel-ZK-DJ}
  (\ZZ_{\posetB})_{T}
  = \bigPolyProd{\posetB}{(\DD^{2})_{\SimpSOne}}{(\SS^{1})_{\SimpSOne}}
  \leftarrow \bigPolyProd{\posetB}{(\DD^{2})_{\SimpSOne}}{e_{0}}
  \to \PolyProd{\posetB}{BS^1}{b_{0}} = \SimpDJ_{\posetB}
\end{equation}
are homotopy equivalences, where \(e_{0}\) is the canonical base points of~\(ES^{1}\).

\begin{proposition}
  \label{thm:ZZSigma-DJSigma-toric-morph}
  The homotopy equivalences~\eqref{eq:Borel-ZK-DJ} connecting \((\ZZ_{\posetB})_{T}\) and~\(DJ_{\posetB}\)
  commute with maps induced by toric morphisms of simplicial posets,
  and so does the map~\(DJ_{\posetB}\to BT\).
\end{proposition}

\begin{proof}
  This follows by applying \Cref{thm:pp-toric-natural} to the maps of pairs used in~\eqref{eq:Borel-ZK-DJ}.
\end{proof}

\section{Classifying spaces of tori}
\label{sec:BT}

\subsection{Review of the formality map for~\texorpdfstring{\(BT\)}{BT}}
\label{sec:formality-BT}

\def\nn{m}

Let \(\Sl\) be a symmetric coalgebra on cogenerators~\(y_{1}\),~\dots,~\(y_{\nn}\) of degree~\(2\).
The dual~\(\Su\) of~\(\Sl\) then is a polynomial algebra~\(\kk[t_{1},\dots,t_{\nn}]\) on generators~\(t_{i}\) dual to the \(y_{i}\)'s.
In this section we review and mildly generalize the construction of a dga morphism
\begin{equation}
  f^{*}\colon C^{*}(BT) \to \Su
\end{equation}
where \(T\) is a torus and
\begin{equation}
  \pi\colon ET\to BT
\end{equation}
the simplicial construction of its universal bundle.
As in~\cite[Sec.~5]{Franz:gersten},
by a torus we mean a compact torus~\(\cong(\SS^{1})^{\nn}\), an algebraic torus~\(\cong(\C^{\times})^{\nn}\)
or a simplicial torus, that is, the simplicial classifying space~\(BN\) of a lattice~\(N\cong\Z^{\nn}\).

The elements of the canonical \(\kk\)-basis for~\(\Sl\) are written as \(y_{\alpha}\), indexed by a multi-index~\(\alpha\in\N^{\nn}\).
If a component~\(\alpha_{i}\) of~\(\alpha\) is non-zero, we write \(\alpha|i\) for the multi-index that is obtained from it
by decreasing the \(i\)-th component by~\(1\).

Let \(\Ll\) be the (strictly) exterior algebra on generators~\(x_{1}\),~\dots,~\(x_{\nn}\) of degree~\(1\).
The differential of the Koszul complex
\begin{equation}
  \Kl = \Ll \otimes \Sl
\end{equation}
is given by
\begin{equation}
  d(a\otimes y_{\alpha}) = \sum_{\alpha_{i}>0} x_{i}\,a\otimes y_{\alpha|i}.
\end{equation}
We often write \(a\,y\) instead of~\(a\otimes y\in\Kl\), and also \(y_{0}=1\).

We first define a map
\begin{equation}
  F\colon \Kl \to C(ET).
\end{equation}
Let \(c_{1}\),~\dots,~\(c_{\nn}\in C_{1}(T)\) be linear combination of loops (hence cycles).
Recall that a loop in~\(T\) is a \(1\)-simplex whose vertices are both equal to~\(1\in T\).
We recursively set
\begin{alignat}{2}
  \label{eq:def-f-0}
  F(1) &= e_{0},\\
  \label{eq:def-f-S}
  F(y) &= S\,F(d y) &\qquad& \text{if \(|y| > 0\)}, \\
  \label{eq:def-f-a}
  F(x_{i_{1}}\cdots x_{i_{k}}\, y) &= c_{i_{1}}\cdots c_{i_{k}}\cdot F(y)
\end{alignat}
for~\(y\in\Sl\) and~\(1\le i_{1}<\dots<i_{k}\le\nn\).
See the following section for the contracting homotopy~\(S\) of~\(C(ET)\)
as well as for the action of the augmented dga~\(C(T)\) on~\(C(ET)\) in~\eqref{eq:def-f-a}.
We note that \(F\) is independent of the order of the basis~\((x_{i})\) because \(C(T)\) is (graded) commutative.

The map~\(F\) induces a map
\begin{equation}
  f\colon \Sl = \kk \otimes_{\Ll} \Kl \to C(BT),
  \quad
  y \mapsto \pi_{*}\,F(y).
\end{equation}
The last claim in the following result explains the name ``formality map'' for~\(f^{*}\).
By abuse of terminology, we use it for any map~\(f^{*}\) constructed as described above.

\begin{theorem}
  \label{thm:f-dgc}
  The map~\(f\) is a morphism of dgcs, hence its dual~\(f^{*}\) a morphism of hgas.
  If \(\Sl=H(BT)\), \(\Ll=H(T)\) and \(c_{1}\),~\dots,~\(c_{\nn}\) represent a basis for~\(H_{1}(T)\)
  that under transgression corresponds to the basis~\(y_{1}\),~\dots,~\(y_{\nn}\) for~\(H_{2}(BT)\),
  then \(f\) and~\(f^{*}\) induce the identity in (co)homology.
\end{theorem}

In fact, \(f\) is a morphism of homotopy Gerstenhaber coalgebras (hgcs),
but we are not using this notion in this paper.

\begin{proof}
  See~\cite[Thm.~5.3]{Franz:gersten}.
  The additional assumptions on~\(\Sl\) and the cycles~\(c_{i}\) stated above
  are made throughout in~\cite[Sec.~5]{Franz:gersten}, but they are not needed to
  establish the (co)multiplicativity of~\(f\) and~\(f^{*}\).
\end{proof}

\subsection{The main ingredients}

We also review the main ingredients of the proof of \Cref{thm:f-dgc} since we will use them
again in the next sections.

The map~\(S\colon C(ET)\to C(ET)\) appearing in~\eqref{eq:def-f-S}
is a homotopy from the projection to the basepoint~\(e_{0}\) to the identity map of~\(ET\).
We therefore have
\begin{equation}
  \label{eq:homotopy-S}
  (d S + S d)(e) =
  \begin{cases}
    e-e_{0} & \text{if \(p=0\),} \\
    e & \text{if \(p>0\)}
  \end{cases}
\end{equation}
for any \(p\)-simplex~\(e\in ET\).
It moreover satisfies the identities
\begin{equation}
  \label{eq:properties-S}
  S S=0 \quad\text{and}\quad S e_{0}=0,
\end{equation}
see~\cite[Sec.~2.7]{Franz:gersten}.

Let \(k\),~\(\ell\geq 0\), and consider a surjection~\(u\colon \{1,2,\ldots, k+\ell\}\to \{1,2,\ldots, k\}\)
satisfying \(u(i)\neq u(i+1)\) for~\(1\leq i<k+\ell\). For any simplicial set~\(X\),
the interval cut operation~\(\AWu{u}\colon C(X)\to  C(X)^{\otimes \ell}\) associated to~\(u\) is defined by
\begin{equation}
  \label{eq:def-AW-u}
  \AWu{u}(x)=\sum \pm x(c_1)\otimes \cdots \otimes x(c_\ell)
\end{equation}
for a \(p\)-simplex~\(x\in X_{p}\), where the sum is over all indices
\begin{equation}
  0=p_0\leq p_1\leq \cdots \leq p_{k+\ell}=p,
\end{equation}
and \(c_i\) is the concatenation of all intervals~\([p_{s-1},p_s]\) such that \(u(i)=s\).
The sign not specified in~\eqref{eq:def-AW-u} 
is the product of a permutation sign and a position sign, compare~\cite[Sec.~3]{Franz:gersten}.
We will only need the diagonal of the dgc~\(C(X)\),
\begin{align}
  \AWu{(12)} = \Delta \colon C(X) \to C(X)\otimes C(X), \\
  \shortintertext{the diagonal with transposed factors,}
  \AWu{(21)} = T_{C(X),C(X)}\,\Delta \colon C(X) \to C(X)\otimes C(X) \\
  \shortintertext{and the map}
  \AWu{(121)} \colon C(X) \to C(X)\otimes C(X)
\end{align}
whose transpose is the \(\cupone\)-product, meaning that for~\(\alpha\),~\(\beta\in C^{*}(X)\) we have
\begin{equation}
  \label{eq:def-cupone}
  \alpha\cupone\beta = -\AWu{(121)}^{*}(\alpha\otimes\beta).
\end{equation}
We recall that \(\AWu{(121)}\) vanishes on \(0\)-simplices.

For~\(X=ET\) these maps interact with the homotopy~\(S\) as follows:
\begin{align}
  \label{eq:diag-S}
  \Delta\, S &= (S\otimes 1)\,\Delta + e_{0}\otimes S, \\
  \label{eq:AW21-S}
  \AWu{(21)}\,S &= (1\otimes S)\,\AWu{(21)} + S\otimes e_{0}, \\
  \label{eq:AW121-S}
  \AWu{(121)}\,S &= - (S \otimes 1)\,\AWu{(121)} + (S\otimes 1)\,\AWu{(21)}\,S,
\end{align}
see~\cite[Lemma~3.5]{Franz:gersten}.

Again for~\(X=ET\) we also use the partial projections
\begin{align}
  \label{eq:def-hatDelta}
  \hat\Delta = (1\otimes\pi_{*})\,\Delta &\colon C(ET) \to C(ET) \otimes C(BT), \\
  \AWhatu{(121)} = (1\otimes\pi_{*})\,\AWu{(121)} &\colon C(ET) \to C(ET) \otimes C(BT).
\end{align}
They are \(C(T)\)-equivariant where \(C(T)\) acts on the factor~\(C(ET)\) of the target only,
see~\cite[Cor.~3.4]{Franz:gersten}.

Since \(T\) acts on~\(ET\) from the left, we get an induced action
of the dg~bialgebra~\(C(T)\) on~\(C(ET)\). The shuffle map underlying the latter action
is a morphism of dgcs, see~\cite[(17.6)]{EilenbergMoore:1966}.
Hence for a loop~\(c\in T\) and a chain~\(w\in C(ET)\) we have
\begin{align}
  \label{eq:equiv-diag-1}
  \Delta(c\cdot w) &= \Delta\,c\cdot \Delta\,w = \bigl(c\otimes 1 + 1\otimes c\bigr)\cdot \Delta\,w \\
  \shortintertext{Likewise, for a \(2\)-simplex~\(b\in T\) whose \(1\)-dimensional front face~\(b^{1}\) and back face~\(b^{2}\) are loops, we get}
  \label{eq:equiv-diag-2}
  \Delta(b\cdot w) &= \bigl(b\otimes 1 + b^{1}\otimes b^{2} + 1\otimes b\bigr)\cdot \Delta\,w \\
  \shortintertext{and therefore also}
  \label{eq:equiv-AW21-1}
  \AWu{(21)}(c\cdot w) &= \bigl(c\otimes 1 + 1\otimes c\bigr)\cdot \AWu{(21)}\,w, \\
  \label{eq:equiv-AW21-2}
  \AWu{(21)}(b\cdot w) &= \bigl(b\otimes 1 - b^{2}\otimes b^{1} + 1\otimes b\bigr)\cdot \AWu{(21)}\,w.
\end{align}

\subsection{A homotopy between different formality maps}

A different choice of linear combinations of loops~\(\tilde{c}_{1}\),~\dots,~\(\tilde{c}_{\nn}\in C_{1}(T)\)
leads to a different map~\(\tilde{f}\colon\Kl\to C(ET)\),
hence to a different map~\(\tilde{f}^{*}\colon C^{*}(BT)\to\Su\).
We need to know how \(f^{*}\) and~\(\tilde{f}^{*}\) are related
if \(c_{i}\) and~\(\tilde{c}_{i}\) are homologous for each~\(1\le i\le\nn\).
The proofs in this section follow the same strategy as those for~\(F\) and~\(f\) in~\cite[Sec.~5]{Franz:gersten}.

\begin{remark}
  One situation where different representatives~\(\tilde{c}_{i}\) come up is the following:
  Any decomposition of~\(T\) into circles determines bases~\(x_{1}\),~\dots,~\(x_{m}\) of~\(H_{1}(T)\)
  and \(y_{1}\),~\dots,~\(y_{m}\in H_{2}(BT)\). It also determines canonical representatives~\(c_{i}\) of the~\(x_{i}\)'s:
  In the simplicial setting, \(c_{i}=[x_{i}]\in BH_{1}(T)\). Topologically,
  this corresponds to a loop that winds once around the (positively oriented) \(i\)-th circle factor with constant speed.
  Altogether we get canonical maps~\(F\colon\Kl\to C(ET)\) and~\(f\colon\Sl\to C(BT)\).
  
  A different decomposition of~\(T\) into circles likewise leads to different bases~\(x'_{i}\) and~\(y'_{i}\)
  as well as representatives~\(c'_{i}\) of the~\(x'_{i}\)'s, hence to different maps~\(F'\) and~\(f'\).
  The~\(x_{i}\)'s and~\(y_{i}\)'s can be expressed as integer linear combinations
  of the~\(x'_{j}\)'s and~\(y'_{i}\)'s, respectively.
  Defining \(\tilde{c}_{i}\) as the corresponding linear combination
  of the~\(c'_{j}\)'s gives different representatives of the~\(x_{i}\)'s, hence another map~\(\tilde{F}\colon\Kl\to C(ET)\).
  A look at the definition~\eqref{eq:def-f-0}--\eqref{eq:def-f-a} of~\(F\) shows that \(\tilde{F}\) agrees with~\(F'\)
  and therefore the induced map~\(\tilde{f}\colon\Sl\to C(BT)\) with~\(f'\).
\end{remark}

{
\def\cc{\tilde{c}}
\def\tildecc{c}
\def\ff{\tilde{F}}
\def\ffbar{\tilde{f}}
\def\tildeff{F}
\def\tildeffbar{f}
\def\hh{H}
\def\hhbar{h}

Let us choose chains~\(b_{1}\),~\dots,~\(b_{\nn}\in C_{2}(T)\) such that
\begin{equation}
  \label{eq:def-bi}
  d\,b_{i} = \cc_{i} - \tildecc_{i}
\end{equation}
for each~\(1\le i\le\nn\).
Based on the chains~\(b_{i}\) and~\(\tildecc_{i}\) and the morphism~\(\ff\), we recursively define a map
\begin{equation}
  \hh\colon \Kl\to C(ET)
\end{equation}
by setting, for~\(y\in\Sl\) and~\(1\le i\le\nn\),
\begin{alignat}{2}
  \label{eq:def-h-0}
  \hh(1) &= 0,\\
  \label{eq:def-h-S}
  \hh(y) &= - S\,\hh(d\,y) &\qquad& \text{if \(|y| > 0\)}, \\
  \label{eq:def-h-a}
  \hh(x_{i}\,a\,y) &= b_{i}\cdot\ff(a\,y) -\tildecc_{i}\cdot \hh(a\,y)
   &\qquad& \text{if \(a\in\bigwedge(x_{1},\dots,x_{i-1})\).}
\end{alignat}
The last line allows to split off generators~\(x_{i}\) one after the other
in decreasing order. For example, one has
\begin{align}
  \hh(x_{2}\,x_{1}\,y)
  &= b_{2}\cdot\ff(x_{1}\,y) - \tildecc_{2}\cdot \hh(x_{1}\,y) \\
  \notag &= b_{2}\,\cc_{1}\cdot \ff(y) - \tildecc_{2}\,b_{1}\cdot\ff(y) + \tildecc_{2}\,\tildecc_{1}\cdot \hh(y).
\end{align}

\begin{lemma}
  \label{thm:H-homotopy}
  The map~\(\hh\) is a homotopy from~\(\tildeff\) to~\(\ff\).
\end{lemma}

\begin{proof}
  Let \(y\in\Sl\).
  For~\(y=1\) we have
  \begin{equation}
    (d\,\hh+\hh\,d)(y)=0=\ff(y)-\tildeff(y).
  \end{equation}
  For~\(\deg{y}>0\) we have by~\eqref{eq:homotopy-S} and induction
  \begin{align}
    d(\hh)(y) &= d\,\hh(y)+\hh(d\,y) = -d\,S\,\hh(d\,y)+\hh(d\,y) = S\,d\,\hh(d\,y) \\
    \notag &= S(\ff-\tildeff-\hh\,d)(d\,y)
    = S\,\ff(d\,y)-S\,\tildeff(d\,y) = \ff(y)-\tildeff(y).
  \end{align}
  For~\(1\le i\le\nn\) and \(a\in\bigwedge(x_{1},\dots,x_{i-1})\) we have
  \begin{align}
    d\,\hh(x_{i}\,a\,y) &= d\bigl(b_{i}\cdot\ff(a\,y) - \tildecc_{i}\cdot \hh(a\,y) \bigr) \\
    \notag &= d\,b_{i}\cdot \ff(a\,y) + b_{i}\cdot d\,\ff(a\,y) + \tildecc_{i}\cdot d\,\hh(a\,y) \\
    \notag &= b_{i}\cdot\ff(d(a\,y)) + (\cc_{i} - \tildecc_{i})\cdot\ff(a\,y) + \tildecc_{i}\cdot(\ff-\tildeff-\hh d)(a\,y) \\
    \notag &= b_{i}\cdot\ff(d(a\,y)) + \cc_{i}\cdot\ff(a\,y) - \tildecc_{i}\cdot\tildeff(a\,y) - \tildecc_{i}\cdot \hh(d(a\,y)) \\
    \notag &= \ff(x_{i}\,a\,y) - \tildeff(x_{i}\,a\,y) + \hh(x_{i}\,d(a\,y)) \\
    \notag &= \ff(x_{i}\,a\,y) - \tildeff(x_{i}\,a\,y) - \hh d(x_{i}\,a\,y).
  \end{align}
  This completes the proof.
\end{proof}

In analogy with~\(f\) we introduce the map
\begin{equation}
  \hhbar\colon \Sl=\kk\otimes_{\Ll}\Kl \to C(BT),
  \quad
  y\mapsto \pi_{*}\,H(y).
\end{equation}
We also define the ``skewed diagonal''
\begin{equation}
  \hat\Delta=(1\otimes\pi_{*})\,\Delta\colon\Kl\to\Kl\otimes\Sl.
\end{equation}
It is \(\Ll\)-equivariant where \(\Ll\) acts only on the first factor of the target,
similarly to the \(C(T)\)-equivariance of the map~\(\hat\Delta\colon C(ET)\to C(ET)\otimes C(BT)\) from~\eqref{eq:def-hatDelta}.

\begin{lemma}
  \label{thm:coalg-homotopy}
  The map~\(\hh\) satisfies
  \begin{equation*}
    \hat\Delta\, \hh = (\tildeff\otimes\hhbar + \hh\otimes\ffbar)\,\hat\Delta.
  \end{equation*}
\end{lemma}

\begin{proof}
  \def\ysum{\!\!\!\beta+\gamma=\alpha\!\!\!}
  \def\yone{y_{\beta}}
  \def\ytwo{y_{\gamma}}
  Let \(y\in\Sl\).
  The claim is trivial for~\(y=1\). For~\(\deg{y}>0\) we have by~\eqref{eq:diag-S} and induction that
  \begin{align}
    \hat\Delta\,\hh(y) &= - \hat\Delta\,S\,\hh(d\,y)
    = - e_{0}\otimes\pi_{*}\,S\,\hh(d\,y) - (S\otimes1)\,\hat\Delta\,\hh(d\,y) \\
    \notag &= \tildeff(1)\otimes \hh(y)
      - (S\otimes1)(\tildeff\otimes\hhbar+\hh\otimes\ffbar)\,\hat\Delta(d\,y) \\
    \notag &= \tildeff(1)\otimes \hh(y)
      - \sum_{\ysum} \bigl(S\,\tildeff\otimes\hhbar+S\,\hh\otimes\ffbar)\,(d\,\yone \otimes \ytwo ) \\
    \notag &= \tildeff(1)\otimes \hh(y)
      + \sum_{\ysum} S\,\tildeff(d\,\yone )\otimes \hhbar(\ytwo )
      - \sum_{\ysum} S\,\hh(d\,\yone )\otimes \ffbar(\ytwo ) \\
    \notag &= \tildeff(1)\otimes \hh(y)
      + \sum_{\substack{\ysum\\\beta\ne0}} \tildeff(\yone )\otimes \hhbar(\ytwo )
      + \sum_{\substack{\ysum\\\beta\ne0}} \hh(\yone )\otimes \ffbar(\ytwo ) \\
    \notag &= \sum_{\ysum} \tildeff(\yone )\otimes \hhbar(\ytwo )
      + \sum_{\ysum} \hh(\yone )\otimes \ffbar(\ytwo ) \\
    \notag &= (\tildeff\otimes\hhbar+ \hh\otimes\ffbar)\,\hat\Delta(y).
  \end{align}
  For~\(1\le i\le\nn\) and \(a\in\bigwedge(x_{1},\dots,x_{i-1})\) we have
  \begin{align}
    \hat\Delta\,\hh(x_{i}\,a\,y) &= \hat\Delta\bigl(b_{i}\cdot \ff(a\,y) - \tildecc_{i}\cdot \hh(a\,y)\bigr) \\
    \notag
    &= (b_{i}\otimes1)\cdot\hat\Delta\,\ff(a\,y) - (\tildecc_{i}\otimes1)\cdot\hat\Delta\,\hh(a\,y) \\
    \notag
    &= (b_{i}\otimes1)\cdot(\ff\otimes\ffbar)\,\hat\Delta(a\,y) - (\tildecc_{i}\otimes1)\cdot(\tildeff\otimes\hhbar+ \hh\otimes\ffbar)\,\hat\Delta(a\,y) \\
    \notag &= (b_{i}\otimes1)\sum_{\ysum}\ff(a\,\yone )\otimes\ffbar(\ytwo ) \\*
    \notag
      &\quad - (\tildecc_{i}\otimes1)\sum_{\ysum}\Bigl((-1)^{\deg{a}}\tildeff(a\,\yone )\otimes\hhbar(\ytwo )+\hh(a\,\yone )\otimes\ffbar(\ytwo )\Bigr) \\
    \notag
    &= \sum_{\ysum}(-1)^{\deg{x_{i} a}}\tildeff(x_{i}\,a\,\yone )\otimes\hhbar(\ytwo ) \\*
    \notag
      &\qquad + \sum_{\ysum} \bigl(b_{i}\cdot\ff(a\,\yone ) -\tildecc_{i}\cdot \hh(a\,\yone ) \bigr)\otimes\ffbar(\ytwo ) \\
    \notag
    &= \bigl(\tildeff\otimes\hhbar + \hh\otimes\ffbar\bigr)\,\hat\Delta(x_{i}\,a\,y). \qedhere
  \end{align}
\end{proof}

\begin{proposition} \( \)
  \label{thm:homotopy-BT}
  The map~\(\hhbar\colon\Sl\to C(BT)\)
  is a coalgebra homotopy from~\(\tildeffbar\) to~\(\ffbar\),
  hence its transpose~\(\hhbar^{*}\colon C^{*}(BT)\to\Sl\)
  an algebra homotopy from~\(\tildeffbar^{*}\) to~\(\ffbar^{*}\).
\end{proposition}

\begin{proof}
  The first part follows from~\eqref{eq:def-h-0}, \Cref{thm:H-homotopy} and \Cref{thm:coalg-homotopy}
  by projecting to~\(C(BT)\) and~\(C(BT)\otimes C(BT)\), respectively, and then passing to~\(\Sl\).
  \Cref{thm:dual-coalg-h} gives the second claim.
\end{proof}

}

\subsection{The homotopy and cup-\texorpdfstring{\(1\)}{1} products}

We now investigate how the algebra homotopy~\(h^{*}\) constructed in the previous section interacts with \(\cupone\)-products.

\def\ll{l}

Let \(\SC\subset C(ET)\) be the graded submodule generated by all chains (in fact, cycles) of the form
\begin{equation}
  \label{eq:def-SC}
  S\,c_{\ll}\,S\,c_{\ll-1} \cdots S\,c_{1}\,e_{0}
\end{equation}
where \(\ll\ge0\) and \(c_{1}\),~\dots,~\(c_{\ll}\) are loops in~\(T\).

\begin{lemma} \( \)
  \label{thm:SC}
  \begin{enumroman}
  \item
    \label{thm:SC-1}
    \(S(\SC)=0\),
  \item
    \label{thm:SC-2}
    \(\Delta(\SC) = \AWu{(12)}(\SC) \subset \SC\otimes\SC\) and \(\AWu{(21)}(\SC) \subset \SC\otimes\SC\),
  \item
    \label{thm:SC-3}
    \(\AWhatu{(121)}(\SC)=0\).
  \end{enumroman}
\end{lemma}

\begin{proof}
  Part~\ref{thm:SC-1} follows from~\eqref{eq:properties-S}.
  For part~\ref{thm:SC-2} it suffices to prove the first claim. This is done by induction on the length~\(\ll\)
  of an expression as in~\eqref{eq:def-SC}, the case~\(\ll=0\) being trivial.
  For larger~\(\ll\) we consider an element~\(S\,c\,z\in\SC\) with~\(z\in\SC\).
  By~\eqref{eq:diag-S},~\eqref{eq:equiv-diag-1}, induction and part~\ref{thm:SC-1} we have
  \begin{align}
    \Delta(S\,c\,z) &= (S\otimes 1)\,\Delta(c\,z) + e_{0}\otimes S\,c\,z \\
    \notag &= (S\,c\otimes 1)\,\Delta\,z + (S\otimes c)\,\Delta\,z + e_{0}\otimes S\,c\,z \\
    \notag &= (S\,c\otimes 1)\,\Delta\,z + e_{0}\otimes S\,c\,z \in \SC\otimes\SC.
  \end{align}
  Part~\ref{thm:SC-3} is another induction on~\(\ll\) with a trivial start. This time we have
  \begin{align}
    \AWhatu{(121)}(S\,c\,z) &= - (S\otimes 1)\,\AWhatu{(121)}\,c\,z + (S\otimes \pi_{*})\,\AWu{(21)}(S\,c\,z) \\
    \notag &= - (S\,c\otimes 1)\,\AWhatu{(121)}\,z = 0,
  \end{align}
  where we have used \eqref{eq:AW121-S}, equivariance of~\(\AWhatu{(121)}\),
  parts~\ref{thm:SC-1} and~\ref{thm:SC-2} of the \namecref{thm:SC} and induction.
\end{proof}

Let us write each chain~\(b_{i}\in C_{2}(T)\) introduced in~\eqref{eq:def-bi} as a sum
\begin{equation}
  \label{eq:def-bij}
  b_{i} = \sum_{s} \lambdavs{i}{s}\,\bis{i}{s}
\end{equation}
of \(2\)-simplices~\(\bis{i}{s}\) in~\(T\) with coefficients~\(\lambdavs{i}{s}\in\kk\).
Since \(T\) is connected, we may assume that all three vertices of each~\(\bis{i}{s}\) are equal to~\(1\in T\). (This is automatic for simplicial tori.)
The \(1\)-di\-men\-sional front face~\(\bis{i}{s}^{1}\) and back face~\(\bis{i}{s}^{2}\) of~\(\bis{i}{s}\) are then again loops.

\begin{lemma}
  The following identities hold for any~\(y\in\Sl\).
  \label{thm:SC-FH}
  \begin{enumroman}
  \item
    \label{thm:SC-FH-1}
    \(F(y)\in\SC\),
  \item
    \label{thm:SC-FH-2}
    \((1\otimes S)\,\AWu{(21)}\,H(y)=0\),
  \item
    \label{thm:SC-FH-3}
    \((S\otimes 1)\,\AWu{(21)}\,H(y) \in \SC\otimes\SC\),
  \item
    \label{thm:SC-FH-4}
    \(\AWhatu{(121)}\,H(y) \in \SC\otimes\pi_{*}\,\SC\),
  \item
    \label{thm:SC-FH-5}
    \(\bigl(\AWhatu{(121)}\otimes 1\bigr)\,\AWhatu{(121)}\,H(y)=0\).
  \end{enumroman}
\end{lemma}

\begin{proof}
  We assume \(y=y_{\alpha}\) with~\(\alpha\in\N^{\nn}\).
  All parts are inductions on~\(\deg{\alpha}\) which are trivial for~\(\alpha=0\).

  For~\ref{thm:SC-FH-1} we have by the definition~\eqref{eq:def-f-0}--\eqref{eq:def-f-a} of~\(F\) that
  \begin{equation}
    F(y_{\alpha}) = \sum_{\alpha_{i}>0} S\,F(x_{i}\,y_{\alpha|i}) = \sum_{\alpha_{i}>0} S\,c_{i}\,F(y_{\alpha|i}) \in \SC.
  \end{equation}

  \def\cc{\tilde{c}}
  \def\tildecc{c}
  \def\ff{\tilde{F}}
  \def\ffbar{\tilde{f}}
  \def\tildeff{F}
  \def\tildeffbar{f}

  Using the definition of~\(H\) and the identity~\eqref{eq:AW21-S}, we have
  \begin{align}
    \label{eq:AW21H}
    \AWu{(21)}\,H(y_{\alpha}) &= - \AWu{(21)}\,S\,H(d\,y_{\alpha}) \\
    \notag &= - (1\otimes S)\,\AWu{(21)}\,H(d\,y_{\alpha}) - S\,H(d\,y_{\alpha}) \otimes e_{0}.
  \end{align}
  Together with~\eqref{eq:properties-S} this proves part~\ref{thm:SC-FH-2}.

  We now turn to part~\ref{thm:SC-FH-3}.
  From~\eqref{eq:AW21H},~\eqref{eq:properties-S} and the definition of~\(H\) we get
  \begin{align}
    \MoveEqLeft{(S\otimes 1)\,\AWu{(21)}\,H(y_{\alpha}) = - \sum_{\alpha_{i}>0} (S\otimes S)\,\AWu{(21)}\,H(x_{i}\,y_{\alpha|i}) } \\
    \notag ={} & \sum_{\alpha_{i}>0} (S\otimes S)\,\AWu{(21)}\bigl(\tildecc_{i}\,H(y_{\alpha|i}) - b_{i}\,\ff(y_{\alpha|i})\bigr) \\
  \shortintertext{which by~\eqref{eq:equiv-AW21-1} and~\eqref{eq:equiv-AW21-2} gives}
    \notag ={} & \sum_{\alpha_{i}>0} \bigl( S\,\tildecc_{i}\otimes S + S\otimes S\,\tildecc_{i} \bigr)\,\AWu{(21)}\,H(y_{\alpha|i}) \\
    \notag & \quad - \sum_{\alpha_{i}>0} \Bigl( S\,b_{i}\otimes S + \sum_{s} \lambdavs{i}{s}\,S\,\bis{i}{s}^{2}\otimes S\,\bis{i}{s}^{1} + S\otimes S\,b_{i} \Bigr)\,\AWu{(21)}\,\ff(y_{\alpha|i}) \\
    \notag ={} & \sum_{\alpha_{i}>0} (S\otimes S\,\tildecc_{i})\,\AWu{(21)}\,H(y_{\alpha|i}) \\
    \notag & \qquad - \sum_{\alpha_{i}>0} \sum_{s} \lambdavs{i}{s}\,(S\,\bis{i}{s}^{2}\otimes S\,\bis{i}{s}^{1})\,\AWu{(21)}\,\ff(y_{\alpha|i})
  \end{align}
  by parts~\ref{thm:SC-FH-1} and~\ref{thm:SC-FH-2} as well as \Cref{thm:SC}\,\ref{thm:SC-1}.
  This last expression lies in~\(\SC\otimes\SC\) by induction, part~\ref{thm:SC-FH-1} and \Cref{thm:SC}\,\ref{thm:SC-2}.

  For part~\ref{thm:SC-FH-4} we have by the definition of~\(H\) and~\eqref{eq:AW121-S} that
  \begin{align}
    \MoveEqLeft{ \AWhatu{(121)}\,H(y_{\alpha}) = - \AWhatu{(121)}\,S\,H(d\,y_{\alpha}) } \\
    \notag &\qquad = (S\otimes 1)\,\AWhatu{(121)}\,H(d\,y_{\alpha}) - (S\otimes\pi_{*})\,\AWu{(21)}\,S\,H(d\,y_{\alpha}) \\
    \notag &\qquad = \sum_{\alpha_{i}>0} (S\otimes 1)\,\AWhatu{(121)}\,H(x_{i}\,y_{\alpha|i}) + (S\otimes\pi_{*})\,\AWu{(21)}\,H(y_{\alpha}) .
  \end{align}
  The second term lies in~\(\SC\otimes\pi_{*}\,\SC\) by part~\ref{thm:SC-FH-3}.
  By the definition of~\(H\) and the \(C(T)\)-equivariance of~\(\AWhatu{(121)}\) each term of the sum can be written as
  \begin{multline}
    (S\otimes 1)\,\AWhatu{(121)}\bigl( b_{i}\,\ff(y_{\alpha|i}) - \tildecc_{i}\,H(y_{\alpha|i}) \bigr) \\
    = (S\,b_{i}\otimes 1)\,\,\AWhatu{(121)}\,\ff(y_{\alpha|i}) + (S\,\tildecc_{i}\otimes 1)\,\AWhatu{(121)}\,H(y_{\alpha|i}).
  \end{multline}
  The first term vanishes by part~\ref{thm:SC-FH-1} together with \Cref{thm:SC}\,\ref{thm:SC-3},
  and the second lies in~\(\SC\otimes\pi_{*}\,\SC\) by induction.

  Part~\ref{thm:SC-FH-5} follows from part~\ref{thm:SC-FH-4} and \Cref{thm:SC}\,\ref{thm:SC-3}.
\end{proof}

Note that from \Cref{thm:SC-FH}\,\ref{thm:SC-FH-1} and \Cref{thm:SC}\,\ref{thm:SC-3}
we can again deduce that \(f^{*}\) annihilates all \(\cupone\)-products, compare \cite[Thm.~5.3\,(ii)]{Franz:gersten}.

\begin{proposition} \( \)
  \label{thm:hdual-cup1-BT}
  \begin{enumroman}
  \item For any~\(\alpha\in C^{*}(BT)\) of even degree we have
    \begin{equation*}
      h^{*}(\alpha) = 0.
    \end{equation*}
  \item \label{item_hcup1} For any~\(\alpha\),~\(\beta\in C^{2}(BT)\) we have
    \begin{equation*}
      h^{*}(\alpha\cupone\beta) = - \sum_{i=1}^{\nn}\sum_{s}\lambdavs{i}{s}\,\bigpair{\alpha,\pi_{*}\,S\,\bis{i}{s}^{2}\,e_{0}}\,\bigpair{\beta,\pi_{*}\,S\,\bis{i}{s}^{1}\,e_{0}}\,t_{i}.
    \end{equation*}
  \item For any~\(\alpha\),~\(\beta\),~\(\gamma\in C^{*}(BT)\) we have
    \begin{equation*}
      h^{*}\bigl((\alpha\cupone\beta)\cupone\gamma\bigr) = 0.
    \end{equation*}
  \end{enumroman}
\end{proposition}

\begin{proof}
  The first claim follows from the fact that \(H^{*}(BT)\) is concentrated in even degrees,
  and the third one is an immediate consequence of \Cref{thm:SC-FH}\,\ref{thm:SC-FH-5}.
  The second one means
  \begin{equation}
    \bigpair{h^{*}(\alpha\cupone\beta),y_{i}} = - \sum_{s}\lambdavs{i}{s}\,\bigpair{\alpha,\pi_{*}\,S\,\bis{i}{s}^{2}\,e_{0}}\,\bigpair{\beta,\pi_{*}\,S\,\bis{i}{s}^{1}\,e_{0}}
  \end{equation}
  for any~\(1\le i\le\nn\). To establish this identity, we need some preparations.

  By~\eqref{eq:AW21-S},~\eqref{eq:properties-S} and~\eqref{eq:equiv-AW21-2} we have
  \begin{align}
    \MoveEqLeft{(S\otimes 1)\,\AWu{(21)}\,S\,b_{i}\,e_{0} = (S\otimes S)\,\AWu{(21)}\,b_{i}\,e_{0}} \\
    \notag &\quad = (S\otimes S)\Bigl(b_{i}\otimes 1 - \sum_{s} \lambdavs{i}{s}\,\bis{i}{s}^{2}\otimes \bis{i}{s}^{1}+1\otimes b_{i}\Bigr)(e_{0}\otimes e_{0}) \\
    \notag &\quad = \sum_{s} \lambdavs{i}{s}\,S\,\bis{i}{s}^{2}\,e_{0}\otimes S\,\bis{i}{s}^{1}\,e_{0}.
  \end{align}
  From this, \eqref{eq:AW121-S} and the equivariant of~\(\AWhatu{(121)}\) we get
  \begin{align}
    \AWhatu{(121)}\,S\,b_{i}\,e_{0} &= - (S\otimes 1)\,\AWhatu{(121)}\,b_{i}\,e_{0} + (S\otimes\pi_{*})\,\AWu{(21)}\,S\,b_{i}\,e_{0} \\
    \notag &= - (S\,b_{i}\otimes 1)\,\AWhatu{(121)}\,e_{0} + (S\otimes\pi_{*})\,\AWu{(21)}\,S\,b_{i}\,e_{0} \\
    \notag &= \sum_{s} \lambdavs{i}{s}\,S\,\bis{i}{s}^{2}\,e_{0}\otimes \pi_{*}\,S\,\bis{i}{s}^{1}\,e_{0}.
  \end{align}

  We have
  \( 
    H(y_{i}) = - S\,h(x_{i}) = - S\,b_{i}\,e_{0}.
  \) 
  Combining the preceding computation with the definition~\eqref{eq:def-cupone} of the~\(\cupone\)-product
  and that of the transpose~\cite[eq.~(2.4)]{Franz:gersten}, we finally get
  \begin{align}
    \bigpair{h^{*}(\alpha\cupone\beta),y_{i}} &= - \bigpair{\alpha\cupone\beta,h(y_{i})} = + \bigpair{\AWu{(121)}^{*}(\alpha\otimes\beta),\pi_{*}\,H(y_{i})} \\
    \notag &= - \bigpair{\alpha\otimes\beta,(\pi_{*}\otimes 1)\,\AWhatu{(121)}\,S\,b_{i}\,e_{0}} \\
    \notag &= - \sum_{s}\lambdavs{i}{s}\,\bigpair{\alpha,\pi_{*}\,S\,\bis{i}{s}^{2}\,e_{0}}\,\bigpair{\beta,\pi_{*}\,S\,\bis{i}{s}^{1}\,e_{0}},
  \end{align}
  as claimed.
\end{proof}

\section{Davis--Januszkiewicz spaces}
\label{sec:DJ}

As in \Cref{sec:morph}, let \(\posetB\) be a simplicial poset on the vertex set~\(V\),
and let \(T=(S^{1})^{V}\) be a simplicial torus of rank~\(|V|=m\).
Suppressing the folding map~\(\posetB\to V\), we write~\(T_{\sigma}=(S^{1})^{\sigma}\subset T\) for~\(\sigma\in\posetB\).
We can then express the simplicial Davis--Januszkiewicz space~\eqref{eq:def-DJ} associated to~\(\posetB\) as
\begin{equation}
  \DJ_{\posetB} = \colim_{\sigma\in\posetB} BT_{\sigma} \,;
\end{equation}
it comes with a canonical map~\(\DJ_{\posetB}\to BT\).

We fix a representative~\(c\in C_{1}(S^{1})\), which leads to canonical representatives~\(c_{v}\in C_{1}(T_{v})\subset C_{1}(T)\) for all~\(v\in V\).
The dgc quasi-iso\-mor\-phisms~\(f_{\sigma}\colon H(BT_{\sigma})\to C(BT_{\sigma})\) given by \Cref{thm:f-dgc}
are natural with respect to inclusions~\(\tau\hookrightarrow\sigma\) of simplices and therefore combine to a map
\begin{equation}
  f_{\posetB,c} \colon H(\DJ_{\posetB}) = \colim_{\sigma\in\posetB}H(BT_{\sigma})
  \to \colim_{\sigma\in\posetB} C(BT_{\sigma}) = C(\DJ_{\posetB}).
\end{equation}
We add the subscript~``\(c\)'' to remind ourselves that the map depends on this representative.
The map~\(f_{\posetB,c}\) is again a quasi-iso\-mor\-phism of dgcs, and its transpose
\begin{equation}
  f_{\posetB,c}^{*}\colon C^{*}(\DJ_{\posetB}) \to H^{*}(\DJ_{\posetB}) = \kk[\posetB]
\end{equation}
a quasi-iso\-mor\-phism of dgas (in fact, of hgas), see~\cite[Thm.~6.2]{Franz:gersten}.

Here \(\kk[\posetB]\) denotes the evenly graded face ring of the simplicial poset~\(\posetB\)
with coefficients in~\(\kk\), see~\cite[Sec.~3.5]{BuchstaberPanov:2015}.
It is the limit of the polynomial rings~\(\kk[\sigma]=H^{*}(BT_{\sigma})\) over all~\(\sigma\in\posetB\).
As a \(\kk\)-algebra, \(\kk[\posetB]\) is generated by elements~\(t_{\sigma}\) corresponding to the simplices~\(\sigma\in\posetB\).
For any~\(\tau\in\posetB\), the restriction of~\(t_{\sigma}\in\kk[\posetB]\) to~\(\kk[\tau]\) equals
\begin{equation}
  \prod_{v\in\sigma} t_{v} \in \kk[\tau] = \kk[\,t_{v}\,|\,v\in\tau\,]
\end{equation}
if \(\sigma\le\tau\) and \(0\) otherwise.
If \(\posetB\) is a simplicial complex, then \(\kk[\posetB]\) is already generated
by the generators~\(t_{v}\) corresponding to the vertices~\(v\) in~\(\posetB\).
Sometimes we also write down a generator~\(t_{v}=0\in\kk[\posetB]\) for a ghost vertex~\(v\in V\).
The canonical map~\(\kk[V]=H^{*}(BT)\to H^{*}(\DJ_{\posetB})=\kk[\posetB]\)
sends \(t_{v}\in\kk[V]\) to~\(t_{v}\in\kk[\posetB]\) for any~\(v\in V\).
A \(\kk\)-basis for~\(\kk[\posetB]\) is given by the standard monomials
\begin{equation}
  t_{\bfsigma}^{\ii} = t_{\sigma_{1}}^{i_{1}}\cdots t_{\sigma_{k}}^{i_{k}}
\end{equation}
where \(\bfsigma=(\sigma_{1}<\dots<\sigma_{k})\) is a chain in~\(\posetB\) of length~\(k\ge0\)
and \(\ii=(i_{1},\dots,i_{k})\) are strictly positive exponents.

Let \((\hatA,\nu)\colon(\hatNA,\posetA)\to(\hatNB,\posetB)\) be a toric morphism of simplicial posets,
and let \((\hatawv{\vB}{\vA})\in\Z^{\VB\times\VB}\) be the matrix representing \(\hatA\).
The former induces a morphism
\begin{equation}
  \hatphi\colon \DJ_{\posetA} \to \DJ_{\posetB}
\end{equation}
between the associated Davis--Januszkiewicz spaces by \Cref{thm:pp-toric-natural}.
For completeness we describe the induced map~\(\hatphi^{*}\colon\kk[\posetB]\to\kk[\posetA]\) between face rings.

\begin{proposition}
  The image under the map~\(\hatphi^{*}\) of the generator~\(t_{\tau}\in\kk[\posetB]\)
  corresponding to a simplex~\(\tau\in\posetB\) is the linear combination
  \begin{equation*}
    \hatphi^{*}(t_{\tau}) = \!\! \sum_{\tau\le\nu(\sigma_{k})} \!\!
    C(\tau;\bfsigma,\ii)\, t_{\sigma_{1}}^{i_{1}}\cdots t_{\sigma_{k}}^{i_{k}}
  \end{equation*}
   of standard monomials in~\(\kk[\posetA]\), where the sum in
  \begin{equation*}
    C(\tau;\bfsigma,\ii) = \sum_{j} \prod_{\vB\in V(\tau)}\hatawv{\vB}{j(\vB)}
  \end{equation*}
  extends over all maps~\(j\colon V(\tau)\to V(\sigma_{k})\) between vertex sets such that
  \begin{equation*}
    j(V(\tau)) = i_{1}\,V(\sigma_{1}) + \dots + i_{k}\,V(\sigma_{k})
  \end{equation*}
  as multi-sets.
  In particular, for a vertex~\(\tau=\vB\) of~\(\posetB\) we have
  \begin{equation*}
    \hatphi^{*}(t_{\vB}) = \sum_{\vA\in\VA} \hatawv{\vB}{\vA}\,t_{\vA},
  \end{equation*}
  and this determines \(\hatphi^{*}\) completely if \(\posetB\) is a simplicial complex.
\end{proposition}

\begin{proof}
  For any simplex~\(\sigma\in\posetA\) the diagram
  \begin{equation}
    \begin{tikzcd}
      H^{*}(\DJ_{\sigma}) & H^{*}(\DJ_{\posetA}) \arrow{l} & H^{*}(B\TA) \arrow{l} \\
      H^{*}(\DJ_{\nu(\sigma)}) \arrow{u}{\hatphi_{\sigma}^{*}} & H^{*}(\DJ_{\posetB}) \arrow{u}{\hatphi^{*}} \arrow{l} & H^{*}(B\TB) \arrow{u}{\hatphi^{*}} \arrow{l}
    \end{tikzcd}
  \end{equation}
  commutes by \Cref{thm:ZZSigma-DJSigma-toric-morph}.
  Taking \(\tau=\vB\) to be a vertex, the commutativity of the right square proves the formula for~\(\hatphi^{*}(t_{\vB})\)
  because we know it to hold in~\(H^{*}(B\TA)\).
  If \(\posetB\) is a simplicial complex, then the elements~\(t_{\vB}\) generate \(\kk[\posetB]\),
  so that \(\hatphi^{*}\) is determined by their images.

  We turn to the general case of a standard monomial~\(t_{\bfsigma}^{\ii}\in\kk[\posetA]\)
  corresponding to~\(\bfsigma=(\sigma_{1}<\dots<\sigma_{k})\) and \(\ii=(i_{1},\dots,i_{k})\).
  By what we have said above, \(t_{\bfsigma}^{\ii}\) restricts to~\(0\in\kk[\sigma]\) unless \(\sigma_{k}\le\sigma\).
  In this latter case the restriction is given by
  \begin{equation}
    \prod_{\vA\in\sigma} t_{\vA}^{j_{\vA}}
  \end{equation}
  where
  \begin{equation}
    \bigcup_{\vA\in\sigma} j_{\vA}\,\{\vA\} = i_{1}\,V(\sigma_{1}) + \dots + i_{k}\,V(\sigma_{k})
  \end{equation}
  as multi-sets.
  
  If \(\tau\not\le\nu(\sigma_{k})\), then by taking \(\sigma=\sigma_{k}\) in the diagram above
  we see that \(t_{\bfsigma}^{\ii}\) does not appear in~\(\hatphi^{*}(t_{\tau})\).
  Otherwise the claimed formula for~\(C(\tau;\bfsigma,\ii)\) follows by multiplying out the product
  \begin{equation}
    \hatphi_{\sigma}^{*}(t_{\tau}) = \prod_{\vB\in\tau} \hatphi_{\sigma}^{*}(t_{\vB}) \in \kk[\sigma].
    \qedhere
  \end{equation}
\end{proof}

\begin{remark}
  \label{rem:pullback}
  In the case of smooth (or just \(\kk\)-smooth) toric varieties, one can identify \(\kk[\posetB]\)
  with the piecewise polynomials on the fan~\(\Sigma\). The generator~\(t_{v}\) associated to the ray~\(v\in\Sigma\)
  corresponds to the piecewise linear `Courant function' that evaluates to~\(1\) on~\(x_{v}\) and to~\(0\)
  on all~\(x_{w}\) with~\(w\ne v\). In this picture, the function~\(\hatphi^{*}\) is the pull-back
  of piecewise polynomials from~\(\SigmaB\) to~\(\SigmaA\), see~\cite[Secs.~1 \& 2]{Brion:1996}.
\end{remark}

The main goal of this section is to establish that the diagram
\begin{equation}
  \label{eq:diag-hstar-DJ}
  \begin{tikzcd}
    C^{*}(\DJ_{\posetB}) \arrow{d}[left]{f_{\posetB,\cB}^{*}} \arrow[dashed]{rd}{h^{*}} \arrow{r}{\hatphi^{*}} & C^{*}(\DJ_{\posetA}) \arrow{d}{f_{\posetA,\cA}^{*}} \\
    H^{*}(\DJ_{\posetB}) \arrow{r}{\hatphi^{*}} & H^{*}(\DJ_{\posetA})
  \end{tikzcd}
\end{equation}
commutative up to an algebra homotopy~\(h^{*}\).

\def\fprime{f}
  
We start by considering the diagram
\begin{equation}
  \begin{tikzcd}
    H(B\TA_{\sigma}) \arrow{d}[left]{f_{\sigma}} \arrow[dashed]{rd}{h_{\sigma}} \arrow{r}{(\hatphi_{\sigma})_{*}} & H(B\TB_{\nu(\sigma)}) \arrow{d}{\fprime_{\nu(\sigma)}} \\
    C(B\TA_{\sigma}) \arrow{r}{(\hatphi_{\sigma})_{*}} & C(B\TB_{\nu(\sigma)})
  \end{tikzcd}
\end{equation}
for a simplex~\(\sigma\in\posetA\).
In this case we obtain a coalgebra homotopy
\begin{equation}
  h_{\sigma}\colon H(B\TA_{\sigma}) \to C(B\TB_{\nu(\sigma)})
\end{equation}
as follows:
For each~\(\vA\in\sigma\) and~\(\tau=\nu(\vA)\) the cycles
\begin{equation}
  \check{c}_{\vA} = (\hatphi_{\vA})_{*}(c_{\vA})
  \qquad\text{and}\qquad
  \hat{c}_{\vA} = \sum_{\vB\in\tau} \hatawv{\vB}{\vA}\,\cB_{\vB}
\end{equation}
in~\(C_{1}(\TB_{\tau})\) represent the same homology class~\((\hatphi_{\sigma})_{*}(x_{\vA})\in H_{1}(\TB_{\tau})\).
We choose some \(b_{\vA}\in C_{2}(\TB_{\tau})\) such that
\begin{equation}
  \label{eq:def-bv}
  d\,b_{\vA} = \check{c}_{\vA} - \hat{c}_{\vA}.
\end{equation}
Note that there is no need to chose an element~\(b_{\vA}\) if \(\vA\) is a ghost vertex
since ghost vertices do not manifest themselves in the associated Davis--Januszkiewicz space.

The composition
\begin{equation}
  \check{f}_{\sigma}=(\hatphi_{\sigma})_{*}\,f_{\sigma}\colon H(B\TA_{\sigma}) \to C(B\TB_{\nu(\sigma)})
\end{equation}
is the formality map determined by the cycles~\(\check{c}_{\vA}\) with~\(\vA\in\sigma\),
and
\begin{equation}
  \hat{f}_{\sigma}=\fprime_{\nu(\sigma)}\,(\hatphi_{\sigma})_{*}
\end{equation}
the one for the cycles~\(\hat{c}_{\vA}\).
Based on the~\(b_{\vA}\)'s for the ordered vertices~\(\vA\in\sigma\), \Cref{thm:homotopy-BT} gives us a coalgebra homotopy~\(h_{\sigma}\)
from~\(\hat{f}_{\sigma}\) to~\(\check{f}_{\sigma}\).

By construction, these homotopies~\(h_{\sigma}\) for all~\(\sigma\in\posetA\)
are natural with respect to inclusion of simplices.
They therefore assemble to a map
\begin{equation}
  \label{eq:def-h-SigmaA}
  h = \colim_{\sigma\in\posetA} h_{\sigma}\colon H(\DJ_{\posetA})
  \to
  C(\DJ_{\posetB}),
\end{equation}
which is again a coalgebra homotopy, namely from~\(f_{\posetB,\cB}\,\hatphi_{*}\) to~\(\hatphi_{*}\,f_{\posetA,\cA}\).
An application of \Cref{thm:dual-coalg-h} finally gives the desired result.

\begin{proposition}
  \label{thm:homotopy-DJ}
  The transpose
  \begin{equation*}
    h^{*}\colon C^{*}(\DJ_{\posetB})\to H^{*}(\DJ_{\posetA})
  \end{equation*}
  of~\eqref{eq:def-h-SigmaA} is an algebra homotopy
  from~\(\hatphi^{*}\,f_{\posetB,\cB}^{*}\) to~\(f_{\posetA,\cA}^{*}\,\hatphi^{*}\)
  as in the diagram~\eqref{eq:diag-hstar-DJ}.
\end{proposition}

As in~\eqref{eq:def-bij}, each chain~\(b_{\vA}\in C_{2}(\TB_{\tau})\) chosen above
can be written as a sum of \(2\)-simplices with coefficients in~\(\kk\),
\begin{equation}
  b_{\vA} = \sum_{s} \lambdavs{\vA}{s}\,\bis{\vA}{s}.
\end{equation}
Since we are in the simplicial setting, each \(1\)-dimensional front face~\(\bis{\vA}{s}^{1}\)
and back face~\(\bis{\vA}{s}^{2}\) of~\(\bis{\vA}{s}\) is a loop in~\(\TB_{\tau}\).
Let us introduce the \(2\)-simplices
\begin{equation}
  \label{eq:def-wvj}
  \wwonevs{\vA}{s} = \piB_{*}\,S\,\bis{\vA}{s}^{1}\,e_{0},
  \qquad\text{and}\qquad
  \wwtwovs{\vA}{s} = \piB_{*}\,S\,\bis{\vA}{s}^{2}\,e_{0}
\end{equation}
in~\(B\TB_{\tau} \subset \DJ_{\posetB} \subset B\TB\).

\begin{corollary}
  \label{thm:hdual-cup1-DJ}
  We continue to use the notation introduced above.
  \begin{enumroman}
  \item For any~\(\alpha\in C^{*}(\DJ_{\posetB})\) of even degree we have
    \begin{equation*}
      h^{*}(\alpha) = 0.
    \end{equation*}
  \item
    \label{thm:hdual-cup1-DJ-2}
    For any~\(\alpha\),~\(\beta\in C^{2}(\DJ_{\posetB})\)
    \begin{equation*}
      h^{*}(\alpha\cupone\beta) = - \sum_{\vA\in\posetA}\sum_{s}\lambdavs{\vA}{s}\,\bigpair{\alpha,\wwtwovs{\vA}{s}}\,\bigpair{\beta,\wwonevs{\vA}{s}}\,t_{\vA}
    \end{equation*}
    where the first sum extends over the non-ghost vertices in~\(\VA\).
  \item For any~\(\alpha\),~\(\beta\),~\(\gamma\in C^{*}(\DJ_{\posetB})\) we have
    \begin{equation*}
      h^{*}\bigl((\alpha\cupone\beta)\cupone\gamma\bigr) = 0.
    \end{equation*}
  \end{enumroman}
\end{corollary}

\begin{proof}
  These are direct translations of \Cref{thm:hdual-cup1-BT} to the present setting.
\end{proof}

\section{The main theorem}
\label{sec:main}

\subsection{Review of the multiplicative isomorphism}

We review material from~\cite{Franz:torprod}.
As in \Cref{sec:DJ} we write \(c\in C_{1}(S^{1})\) for chosen representative of~\(1\in H_{1}(S^{1};\Z)=\Z\).

Let \(\parqdef{\posetB}{\KK}\) be a partial quotient. The Koszul complex
\begin{equation}
  \Kl_{\posetB} = H^{*}(L) \otimes \kk[\posetB]
\end{equation}
has been defined in~\eqref{eq:def-Kl}.
Let \(N=H_{1}(L;\Z)\) be the lattice associated to~\(L\), say of rank~\(n\).
We choose a basis~\((x_{i})\) for~\(N\cong\Z^{n}\). Together with the canonical basis~\((e_{v})_{v\in V}\) for~\(\hatN=\Z^{V}=\Z^{m}\)
this allows to express the projection map~\(T\to L\) by the characteristic matrix~\((\xvi{v}{i})\in\Z^{n\times m}\)
where \(v\in V\) and \(1\le i\le n\).

Note that here we are using the simplicial torus~\(L=BN\) instead of the compact torus~\(\LL\) acting on~\(\parq{\posetB}{\KK}\).
As discussed in~\cite[Sec.~3.3]{Franz:torprod}, the cocycles and cohomology classes of~\(L\) in degree~\(1\) as well as those
of its simplicial classifying space~\(BL\) in degree~\(2\) are naturally isomorphic to the group of additive homomorphism from~\(N\) to~\(\kk\),
\begin{equation}
  \label{eq:iso-deg-1-2}
  \Hom(N,\kk) = Z^{1}(L) = H^{1}(L) = Z^{2}(BL) = H^{2}(BL).
\end{equation}

The isomorphism
\begin{equation}
  \Psi_{\posetB,c}\colon H^{*}(\parq{\posetB}{\KK}) \to H^{*}(\Kl_{\posetB}) = \Tor_{H^{*}(BL)}\bigl(\kk,\kk[\posetB]\bigr)
\end{equation}
is the map induced in cohomology by the following zigzag of quasi-iso\-mor\-phism, recalled below:
\begin{multline}
  \label{eq:quisos-X-Kl}
  C^{*}(\parq{\posetB}{\KK})
  \longrightarrow C^{*}(EK\timesunder{K}\ZZ_{\posetB}) \\
  \longleftarrow \BB\bigl(\kk,C^{*}(BL),C^{*}(ET\timesunder{T}\ZZ_{\posetB})\bigr)
  \longleftrightarrow \BB\bigl(\kk,C^{*}(BL),C^{*}(\DJ_{\posetB})\bigr) \\
  \xrightarrow{\BB(1,1,f_{\posetB,c}^{*})} \BB_{c}\bigl(\kk,C^{*}(BL),\kk[\posetB]\bigr)
  \stackrel{\Phi_{\posetB}}{\longleftarrow} \Kl_{\posetB}.
\end{multline}
Note that we add the subscript~``\(c\)'' to the last bar construction
because its differential depends implicitly on the quasi-iso\-mor\-phism~\(f_{\posetB,c}^{*}\colon C^{*}(\DJ_{\posetB})\to\kk[\posetB]\)
given by \Cref{thm:homotopy-DJ} and therefore on the representative~\(c\in C_{1}(S^{1})\).

The first map in~\eqref{eq:quisos-X-Kl} is induced by the composition
\begin{equation}
  \label{eq:Z-Sigma-K-X-Sigma}
  EK\timesunder{K}\ZZ_{\posetB} \to \EEKK\timesunder{\KK}\ZZ_{\posetB} \to \parq{\posetB}{\KK}
\end{equation}
and natural in the spaces and groups involved, see~\cite[Lemma~4.2, eq.~(4.6)]{Franz:torprod}.
The second map is explained in~\cite[Prop.~3.2]{Franz:torprod} and is again natural.
The arrow~``\(\leftrightarrow\)'' abbreviates the zigzag of quasi-iso\-mor\-phisms
induced by the maps displayed in~\eqref{eq:Borel-ZK-DJ}.

Let \(\alpha_{1}\),~\dots,~\(\alpha_{n}\) be the basis for~\(H^{1}(L;\Z)\) dual to the basis~\((x_{i})\) for~\(N\).
As in~\cite[Sec.~6]{Franz:torprod} we let \(\gamma_i\in C^2(BL)\) for~\(1\le i\le n\)
be the cocycle corresponding to~\(\alpha_i\in H^1(L)\) under the isomorphisms~\eqref{eq:iso-deg-1-2}.
The final arrow in~\eqref{eq:quisos-X-Kl} is the map
\begin{align}
  \label{eq:quiso-Kl-baronesided}
  \Phi_{\posetB}\colon \Kl_{\posetB} = H^{*}(L)\otimes \kk[\posetB] &\to \BB_{c}\bigl(\kk, C^{*}(BL), \kk[\posetB]\bigr) \\
  \notag \alpha_{i_{1}}\cdots\alpha_{i_k}\otimes f &\mapsto [\gamma_{i_{1}}]\circ\dots\circ[\gamma_{i_k}]\otimes f
\end{align}
with~\(k\ge0\) and~\(1\le i_{1}<\dots<i_{k}\le n\). It is a quasi-iso\-mor\-phism of complexes, see~\cite[Prop.~6.2]{Franz:torprod}.
Note that it depends on the chosen isomorphism~\(N\cong\Z^{n}\).
However, as a morphism of graded \(\kk\)-modules it is independent of the representative~\(c\).

The first four maps in~\eqref{eq:quisos-X-Kl} are actually multiplicative.
We will only need that the one-sided bar construction~%
\(\BB_{c}(\kk, C^{*}(BL), \kk[\posetB])\) is a dga with product given by~\eqref{eq:prod-BB-onesided-comm}.

Following \cite[Sec.~7]{Franz:torprod},
we endow \(\Kl_{\posetB}\) with the twisted \(*\)-product determined by~\eqref{eq:twisted-prod}, based on elements
\begin{equation}
  \label{eq:def-qij}
  q_{ij} = f_{\posetB,c}^{*}(\gamma_{ij})
\end{equation}
for certain cochains~\(\gamma_{ij}\in C^{2}(BL)\) where~\(1\le j\le i\le n\).
Then \(\Kl_{\posetB}\) becomes a dga and the map~\(\Phi_{\posetB}\)
multiplicative up to a \(\kk[\posetB]\)-bilinear homotopy
\begin{equation}
  \label{eq:H-Kl-Kl-B}
  H\colon \Kl_{\posetB}\otimes\Kl_{\posetB} \to \BB_{c}\bigl(\kk, C^{*}(BL), \kk[\posetB]\bigr),
\end{equation}
that is,
\begin{equation}
  \label{eq:homotopy-twisted-prod}
  (d\,H+H\,d)(\alpha\otimes \beta)=\Phi_{\posetB}(\alpha)\circ \Phi_{\posetB}(\beta) - \Phi_{\posetB}(\alpha*\beta).
\end{equation}

\begin{remark}
  \label{rem:equiv-formal-prod}
  Recall from~\cite[Lemma~3.1]{Franz:torprod} that
  the equivariant cohomology of~\(\parq{\posetB}{\KK}\) with respect to the topological group~\(\LL\)
  is isomorphic to the one with respect to the simplicial group~\(L\).
  Both are isomorphic to the face algebra~\(\kk[\posetB]\) via the zigzag of quasi-iso\-mor\-phisms
  \begin{equation}
    C^{*}(EL\timesunder{L}\parq{\posetB}{\KK})
    \longrightarrow C^{*}(ET\timesunder{T}\ZZ_{\posetB})
    \longleftrightarrow C^{*}(\DJ_{\posetB})
    \xrightarrow{f_{\posetB,c}^{*}} \kk[\posetB],
  \end{equation}
  which is the equivariant counterpart of~\eqref{eq:quisos-X-Kl} and involves again the maps from~\eqref{eq:Borel-ZK-DJ}.
  (The first map above is induced by the projection~\(ET\timesunder{T}\ZZ_{\posetB}\to EL\timesunder{L}\parq{\posetB}{\KK}\),
  which is a bundle with fibre~\(EK\).) By comparing the two zigzags, one can show that the diagram
  \begin{equation}
    \begin{tikzcd}      
      H_{L}^{*}(\parq{\posetB}{\KK}) \arrow{d}[left]{\cong} \arrow{r} & H^{*}(\parq{\posetB}{\KK}) \arrow{d}{\Psi_{\Sigma,c}} \\
      \kk[\posetB] \arrow{r} & \Tor_{H^{*}(BL)}\bigl(\kk,\kk[\posetB]\bigr)
    \end{tikzcd}
  \end{equation}
  commutes, where the top arrow is the canonical restriction map and the bottom one
  induced by the inclusion~\(\kk[\posetB]\hookrightarrow\Kl_{\posetB}\).

  Now assume that \(\parq{\posetB}{\KK}\) is equivariantly formal in the sense that \(H_{L}^{*}(\parq{\posetB}{\KK})\) surjects onto~\(H^{*}(\parq{\posetB}{\KK})\).
  For example, all smooth compact toric varieties fall into this class, \cf~\cite[Thm.~7.4.35]{BuchstaberPanov:2015}.
  The commutative diagram above then implies that
  any cohomology class in~\(H^*(\parq{\posetB}{\KK})\) can be represented by~\(1\otimes f\in\Kl_{\posetB}\) for some~\(f\in \kk[\posetB]\).
  We therefore conclude that the twisted and untwisted product on~\(\Kl_{\posetB}\) agree in cohomology in this case.
\end{remark}

\begin{example}
  \label{ex:nontrivial-qij-mult}
  We introduce an example that we will revisit throughout \Cref{sec:application}.
  We consider the fan~\(\Sigma\) in~\(\R^{3}\) whose maximal cones are the rays
  through the vectors~\(v=[1,1,1]\) and~\(w=-v=[-1,-1,-1]\).
  Then \(\XXX_{\Sigma}\cong\C^{\times}\times\C^{\times}\times\CP^{1}\), \cf~\cite[p.~22]{Fulton:1993}.
  Recall that the Koszul complex~\(\Kl_{\Sigma}\) is bigraded with an element~\(\alpha\otimes f\)
  having bidegree~\((-\deg{\alpha},\deg{f}+2\deg{\alpha})\). In this example, the bigraded \(\Tor\)~terms are as follows:
    \begin{equation}
    \begin{array}{ccc|c}
      \kk &  & & 6 \\
      \kk & \kk^{2} & & 4 \\
      & \kk^{2} & \kk & 2 \\
      & & \kk & 0 \\
      \hline
      -2 & -1 & 0 &
    \end{array}
  \end{equation}
  We use the canonical basis of~\(N=\Z^{3}\)
  and fix bases for the cohomology in the following bidegrees:
  \begin{align}
    (-1,2) &\colon a_{1} = [\alpha_{1}-\alpha_{3}], \quad a_{2} = [\alpha_{2}-\alpha_{3}], \\
    (-2,4) &\colon b = [\alpha_{1}\,\alpha_{2} + \alpha_{2}\,\alpha_{3} - \alpha_{1}\,\alpha_{3}], \\
    (0,2) &\colon c = [t_{v}] = [t_{w}].
  \end{align}
  The twisting terms for the product are given by
  \begin{equation}\label{eq:twisting terms in example}
    q_{11} = q_{22} = q_{33} = t_{w},
    \qquad
    q_{21} = q_{31} = q_{32} = t_{v}+t_{w}.
  \end{equation}
  For example, we have
  \begin{align}
    a_{1}*a_{2} &= [\alpha_{1}*\alpha_{2}-\alpha_{1}*\alpha_{3}-\alpha_{3}*\alpha_{2}+\alpha_{3}*\alpha_{3}] \\
    \notag &= [\alpha_{1}\,\alpha_{2} - \alpha_{1}\,\alpha_{3} - \alpha_{3}\,\alpha_{2} - q_{32} + q_{33}] = b - c .
  \end{align}

  If we took \(v\) (or~\(w\)) as part of our basis for~\(N\), then all twisting terms would vanish in cohomology.
  See \Cref{ex:nontrivial-qij-basis}, where we compare these two choices.
  As the examples given in~\cite[Ex.~1.1, Sec.~9]{Franz:torprod} show,
  it is not always possible to make the twisting terms vanish this way.
\end{example}

\subsection{The map induced in cohomology by a toric morphism}

Let \(\parqdef{\posetB}{\KK}\) and~\(\parqdef{\posetA}{\KK'}\)
be two partial quotients, and let \((A,\nu)\colon (\NA,\posetA)\to(\NB,\posetB)\)
be a toric morphism of simplicial posets as defined in \Cref{sec:morph}.
It induces the map~\(\phi\colon \parq{\posetA}{\KK'} \to \parq{\posetB}{\KK}\)
and its lift~\(\phi\colon \ZZ_{\posetA}\to \ZZ_{\posetB}\).
The latter is described by~\(\nu\) and the matrix~\(\hatA\in\N^{\VB\times\VA}\).
We also write \(\phi\) for the associated map~\(B\LA\to B\LB\)
as well as \(\hatphi\) for the maps~\(B\TA\to B\TB\) and~\(\DJ_{\posetA}\to\DJ_{\posetB}\).
Moreover, we write
\begin{equation}
  \rho \colon \DJ_{\posetB}\to BT \to BL
\end{equation}
for the map induced by the projection~\(T\to L\); the map~\(\rho'\) is defined analogously.

In this section we are going to show the following diagram induces a commutative diagram in cohomology.
This will be the main technical result of this paper.
\begin{equation}
  \label{diag:main}
  \begin{tikzcd}
    C^{*}(\parq{\posetB}{\KK}) \arrow{d} \arrow{r}{\phi^{*}} & C^{*}(\parq{\posetA}{\KK'}) \arrow{d} \\
    C^{*}((\ZZ_{\posetB})_{\KB}) \arrow{r} & C^{*}((\ZZ_{\posetA})_{\KA}) \\
    \BB\bigl(\kk,C^{*}(B\LB),C^{*}((\ZZ_{\posetB})_{\TB})\bigr) \arrow{u} \arrow[leftrightarrow]{d} \arrow{r} & \BB\bigl(\kk,C^{*}(B\LA),C^{*}((\ZZ_{\posetA})_{\TA})\bigr) \arrow{u} \arrow[leftrightarrow]{d} \\
    \BB\bigl(\kk,C^{*}(B\LB),C^{*}(\DJ_{\posetB})\bigr) \arrow{d}[left]{\BB(1,1,f_{\posetB,\cB}^{*})} \arrow{r}{\BB(\phi)} & \BB\bigl(\kk,C^{*}(B\LA),C^{*}(\DJ_{\posetA})\bigr) \arrow{d}[right]{\BB(1,1,f_{\posetA,\cA}^{*})} \\
    \BB_{\cB}\bigl(\kk,C^{*}(B\LB),\kk[\posetB]\bigr) \arrow{r}{\Theta_{h^{*}}} & \BB_{\cA}\bigl(\kk,C^{*}(B\LA),\kk[\posetA]\bigr) \\
    \Kl_{\posetB} \arrow{u}{\Phi_{\posetB}} \arrow{r}{\hatKtw} & \Kl_{\posetA} \arrow{u}[right]{\Phi_{\posetA}}
  \end{tikzcd}
\end{equation}
The vertical arrows are as in~\eqref{eq:quisos-X-Kl},
based on isomorphisms~\(N\cong\Z^{n}\) and~\(N'\cong\Z^{n'}\)
as well as on representatives~\(c\),~\(c'\in C_{1}(S^{1})\).
The first two squares commute by naturality
and the middle square by \Cref{thm:ZZSigma-DJSigma-toric-morph}.
By abuse of notation we write \(\BB(\phi)\) to denote the map
\begin{equation}
  \BB(\phi) \colon
  \BB\bigl(\kk,C^{*}(B\LB),C^{*}(\DJ_{\posetB})\bigr) \xrightarrow{\BB(1,\phi^{*},\hatphi^{*})} \BB\bigl(\kk,C^{*}(B\LA),C^{*}(\DJ_{\posetA})\bigr).
\end{equation}

The chain map~\(\Theta_{h^{*}}\) in the bottom row of the fourth square is the map~\eqref{eq:def-Theta-h} applied to the diagram
\begin{equation}
  \label{diag:phif-fphi}
  \begin{tikzcd}
    C^{*}(B\LB) \arrow{d}[left]{\rhoB^{*}} \arrow{r}{\phi^{*}} & C^{*}(B\LA) \arrow{d}{\rhoA^{*}} \\
    C^{*}(DJ_{\posetB}) \arrow{d}[left]{f_{\posetB,\cB}^{*}} \arrow[dashed]{dr}{h^{*}} \arrow{r}{\hatphi^{*}} & C^{*}(DJ_{\posetA}) \arrow{d}{f_{\posetA,\cA}^{*}} \\
    \kk[\posetB] \arrow{r}{\hatphi^{*}} & \kk[\posetA]
  \end{tikzcd}
\end{equation}
where \(h^{*}\colon C^{*}(DJ_{\posetB})\to \kk[\posetA]\) is the algebra homotopy
from~\(\hatphi^{*}\,f^{*}_{\posetB}\) to~\(f^{*}_{\posetA}\,\hatphi^{*}\) established in \Cref{thm:homotopy-DJ}.
Remember that it depends not only on~\(\posetA\) and~\(\posetB\),
but also on~\(c\),~\(c'\) and on the chains~\(b_{\vA}\in C_{2}(\TB)\)
chosen to satisfy the condition~\eqref{eq:def-bv} for the vertices~\(\vA\in\posetA\).

The commutativity of the fourth square up to homotopy therefore is a special case of \Cref{thm:homotopy-bar}.
Spelling out the maps suppressed in~\eqref{eq:def-Theta-h},we see that \(\Theta_{h^{*}}\) is of the form
\begin{gather}
  \label{eq:theta_h}
  \Theta_{h^{*}}\colon \BB_{\cB}\bigl(\kk,C^{*}(B\LB),\kk[\posetB]\bigr) \to \BB_{\cA}\bigl(\kk,C^{*}(B\LA),\kk[\posetA]\bigr) \\
  \notag
  [a_{1}|\dots|a_{k}]\otimes f \mapsto
  \begin{cases}
    \BBone \otimes \hatphi^{*}(f) & \text{if \(k=0\)} \\
    \bigl[\phi^{*}(a_{1})|\dots|\phi^{*}(a_{k})\bigr]\otimes \hatphi^{*}(f) \\
    \quad + \bigl[\phi^{*}(a_{1})|\dots|\phi^{*}(a_{k-1})\bigr]\otimes h^{*}\rhoB^{*}(a_{k})\,\hatphi^{*}(f) & \text{if~\(k>0\).}
  \end{cases}
\end{gather}

We introduce the abbreviation
\begin{equation}
  \label{eq:def-hatq}
  \hatq_{ij} = - h^{*}\,\rhoB^{*} (\gamma_{j}\cupone \gamma_{i}) \in \kk[\posetA]
\end{equation}
for~\(1\le j<i\le \nB\).
Recall from~\cite[Sec.~3.3]{Franz:torprod} that the \(2\)-simplices in~\(B\TB\) are
in bijection with elements of the lattice~\(\hatNB\).
In this interpretation, the \(2\)-simplices~\(\wwtwovs{\vA}{s}\) and~\(\wwtwovs{\vA}{s}\) introduced in~\eqref{eq:def-wvj}
have coordinates~\(\wwonevsv{\vA}{s}{\vB}\) and~\(\wwtwovsv{\vA}{s}{\vB}\) with~\(\vB\in\VB\).

\begin{lemma}
  \label{thm:hatq}
  For any~\(1\le j<i\le \nB\) we have
  \begin{equation*}
    \hatq_{ij} = \sum_{\vA\in\posetA} \sum_{\vB,\wB\in\VB} \sum_{s} \lambdavs{\vA}{s}\,\wwtwovsv{\vA}{s}{\vB}\,\wwonevsv{\vA}{s}{\wB}\,\xvi{\vB}{j}\,\xvi{\wB}{i}\,t_{\vA}
  \end{equation*}
  where the first sum extends over the non-ghost vertices in~\(\VA\).
\end{lemma}

\begin{proof}
  This follows from \Cref{thm:hdual-cup1-DJ}\,\ref{thm:hdual-cup1-DJ-2} since
  \begin{equation}
    \pair{\gamma_{i},\wwonevs{\vA}{s}} = \sum_{\vB\in\VB}\wwonevsv{\vA}{s}{\vB}\,\xvi{\vB}{i}
    \qquad\text{and}\qquad
    \pair{\gamma_{j},\wwtwovs{\vA}{s}} = \sum_{\vB\in\VB}\wwtwovsv{\vA}{s}{\vB}\,\xvi{\vB}{j}.
    \qedhere
  \end{equation}
\end{proof}

\def\GammaA{\Gamma'}
To state and prove the next result we define
\begin{align}
  \Gamma\colon H^{*}(\LB) &\to \BB\,C^{*}(B\LB),
  & \alpha_{i_{1}} \cdots \alpha_{i_k} &\mapsto [\gamma_{i_{1}}] \circ \dots \circ [\gamma_{i_k}], \\
  \GammaA \colon H^{*}(\LB) &\to \BB\,C^{*}(B\LA),
  & \alpha_{i_{1}} \cdots \alpha_{i_k} &\mapsto [\phi^{*}(\gamma_{i_{1}})] \circ \dots \circ [\phi^{*}(\gamma_{i_k})]
\end{align}
where \(1\le i_{1}<\dots<i_{k}\le\nB\). Since \(\BB\,C^{*}(B\LB)\) is a dg~bialgebra and
each factor~\([\gamma_{i_s}]\) primitive, the map~\(\Gamma\) is a morphism of dgcs, and so is \(\GammaA = \BB(\phi)\,\Gamma\).

\begin{lemma}
  \label{thm:Thetah-PhiSigma}
  The composition
  \begin{gather*}
    \Theta_{h^{*}}\,\Phi_{\posetB}\colon \Kl_{\posetB}\to \BB_{\cA}\bigl(\kk,C^{*}(B\LA),\kk[\posetA]\bigr) \\
  \shortintertext{is given by}
    \alpha\otimes f \mapsto \GammaA(\alpha)\otimes \hatphi^{*}(f)
    + \sum_{i>j} \GammaA\bigl(\iota(x_i)\,\iota(x_j)\,\alpha\bigr)\otimes\,\hatq_{ij}\,\hatphi^{*}(f).
  \end{gather*}
\end{lemma}

\begin{proof}
  We are using the formula~\eqref{eq:def-Theta-h-bis} for~\(\Theta_{h^{*}}\). Writing
  \begin{equation}
    \alpha_{I} = \alpha_{i_{1}}\cdots \alpha_{i_{k}}
  \end{equation}
  for~\(I=\{i_{1}<\dots<i_{k}\}\subset\{1,\dots,\nB\}\), we have
  \begin{equation}
    \Delta\,\Gamma(\alpha_{I}) = (\Gamma\otimes \Gamma)\,\Delta\,\alpha_{I} = \sum_{I=\Jone\sqcup \Jtwo}(-1)^{\sigma(\Jone,\Jtwo)}\,\Gamma(\alpha_{\Jone})\otimes\Gamma(\alpha_{\Jtwo})
  \end{equation}
  where \(\sigma(\Jone,\Jtwo)\) is the permutation sign of the decomposition~\(I=\Jone\sqcup \Jtwo\).
  
  Let \(t=t_{C^{*}(B\LB)}\colon \BB\,C^{*}(B\LB)\to C^{*}(B\LB)\) be the canonical twisting cochain.
  We know from \Cref{thm:iterated-cup1} that \(t\,\Gamma(\alpha_{\Jtwo})\) is an iterated \(\cupone\)-product, hence
  \(h^{*}\,t\,\Gamma(\alpha_{\Jtwo})\) vanishes unless \(\deg{\Jtwo}=2\) by \Cref{thm:hdual-cup1-DJ}.
  Moreover, in the case where \(\Jtwo=\{j<i\}\) has two elements, we have again by \Cref{thm:iterated-cup1} that
  \begin{equation}
    \label{eq:h rho t hatpsi}
    h^{*}\,\rhoB^{*}\,t\,\Gamma(\alpha_{\Jtwo})
    = h^{*}\,\rhoB^{*}\,t\bigl([\gamma_{j}]\circ[\gamma_{i}]\bigr)
    = - h^{*}\,\rhoB^{*}\bigl(\gamma_{j}\cupone\gamma_{i} \bigr).
  \end{equation}
  Comparison with \Cref{thm:hdual-cup1-DJ}\,\ref{thm:hdual-cup1-DJ-2} shows that the claimed formula holds up to sign.

  \goodbreak

  To conclude the proof, we note that again for~\(\Jtwo=\{i_{q}=j<i_{p}=i\}\) we have
  \begin{equation}
    \sigma(\Jone,\Jtwo) = (\nB-p) + (\nB-q-1) \equiv (q-1) + (p-2) \pmod{2},
  \end{equation}
  hence
  \begin{equation}
    \alpha_{\Jone} = (-1)^{\sigma(\Jone,\Jtwo)}\,\iota(x_i)\,\iota(x_j)\,\alpha_{I}.
  \end{equation}
  This proves that the sign is correct.
\end{proof}

Recall that in~\eqref{eq:def-Ktw} and~\eqref{eq:intro:def-hatKtw} we have defined the map
\begin{align}
  \label{eq:def-hatKtw}
  \hatKtw\colon \Kl_{\posetB} &\to \Kl_{\posetA}, \\
  \notag \alpha\otimes f &\mapsto \Ktw(\alpha\otimes f)
  + \sum_{i>j} \Ktw\bigl(\iota(x_{i})\,\iota(x_{j})\,\alpha \otimes f\bigr)\,\hatq_{ij}
\end{align}
where
\begin{equation}
  \Ktw(\alpha_{i_{1}}\cdots\alpha_{i_{k}}\otimes f) = \phi^{*}(\alpha_{i_{1}})*\dots*\phi^{*}(\alpha_{i_{k}}) * \hatphi^{*}(f)
\end{equation}
for~\(k\ge0\) and~\(i_{1}<\dots<i_{k}\). (Note the twisted products on the right-hand side.)
One can check that \(\Ktw\) as well as all operators~\(\iota(x_{i})\,\iota(x_{j})\colon\Kl_{\posetB}\to\Kl_{\posetB}\)
commute with the differentials, which implies that \(\hatKtw\) is a chain map.

\begin{lemma}
  \label{thm:bottom-square}
  The bottom square in diagram~\eqref{diag:main},
  \begin{equation*}
    \begin{tikzcd}
      \BB_{\cB}\bigl(\kk,C^{*}(B\LB),\kk[\posetB]\bigr) \arrow{r}{\Theta_{h^{*}}} & \BB_{\cA}\bigl(\kk,C^{*}(B\LA),\kk[\posetA]\bigr) \\
      \Kl_{\posetB} \arrow{u}{\Phi_{\posetB}} \arrow{r}{\hatKtw} & \Kl_{\posetA} \arrow{u}[right]{\Phi_{\posetA}} \mathrlap{,}
    \end{tikzcd}
  \end{equation*}
  is homotopy commutative.
\end{lemma}

\begin{proof}
  Based on the homotopy~\(H\) recalled in~\eqref{eq:H-Kl-Kl-B}, we define the map
  \begin{equation}
    \tildeH\colon \Kl_{\posetB} \to \BB_{\cA}\bigl(\kk, C^{*}(B\LA),\kk[\posetA]\bigr)
  \end{equation}
  of degree~\(-1\) by linearity with respect to the algebra map~\(\hatphi^{*}\colon\kk[\posetB]\to\kk[\posetA]\) and
  \begin{align}
    \tildeH(1\otimes 1) &= 0 \\
    \tildeH(\alpha_{i_{1}}\cdots \alpha_{i_{k}}\otimes 1) &=
    H\Big(\phi^{*}(\alpha_{i_{1}})*\dots * \phi^{*}(\alpha_{i_{k-1}})\otimes \phi^{*}(\alpha_{i_{k}})\Big) \\
    \notag &\qquad + \tildeH(\alpha_{i_{1}}\cdots \alpha_{i_{k-1}}) \circ \Phi_{\posetA}(\phi^{*}(\alpha_{i_{k}}))
  \end{align}
  for~\(k\ge 1\) and~\(1\le i_{1}<\dots<i_{k}\le \nB\).
  We claim that \(\tildeH\) is a homotopy from~\(\Phi_{\posetA}\,\Ktw\) to~\(\BB(\phi)\,\Phi_{\posetB}\). In other words,
  \begin{equation}
    d(\tildeH) = \BB(\phi)\,\Phi_{\posetB}-\Phi_{\posetA}\,\Ktw.
  \end{equation}
  The isomorphism~\eqref{eq:iso-deg-1-2} and the definition of~\(\Ktw\) imply that this identity can be written as
  \begin{multline}
    \label{eq:tildeH-final}
    (d\,\tildeH+\tildeH\,d)(\alpha_{i_{1}}\cdots \alpha_{i_{k}}) = \\*
    \Phi_{\posetA}(\phi^{*}(\alpha_{i_{1}}))\circ \dots \circ \Phi_{\posetA}(\phi^{*}(\alpha_{i_{k}}))
    - \Phi_{\posetA}\big(\phi^{*}(\alpha_{i_{1}})*\dots*\phi^{*}(\alpha_{i_{k}})\big).
  \end{multline}

  We verify \eqref{eq:tildeH-final} by induction on~\(k\), the case~\(k=0\) being trivial.
  For~\(k\geq 1\) we obtain from definition of~\(\tildeH\), induction and~\eqref{eq:homotopy-twisted-prod} that
  \begin{align}
    d\,\tildeH(\alpha_{i_{1}} \cdots \alpha_{i_{k}}) &= d\,H\Big(\phi^{*}(\alpha_{i_{1}}) * \dots * \phi^{*}(\alpha_{i_{k-1}})\otimes \phi^{*}(\alpha_{i_{k}})\Big) \\*
    \notag &\qquad + d\,\tildeH(\alpha_{i_{1}} \cdots \alpha_{i_{k-1}}) \circ \Phi_{\posetA}(\phi^{*}(\alpha_{i_{k}})) \\*
    \notag &\qquad + (-1)^{k}\,\tildeH(\alpha_{i_{1}} \cdots \alpha_{i_{k-1}})\circ d\,\Phi_{\posetA}\big(\phi^{*}(\alpha_{i_{k}})\big) \\
    \notag &= - H\,d\Big(\phi^{*}(\alpha_{i_{1}}) * \dots * \phi^{*}(\alpha_{i_{k-1}})\otimes \phi^{*}(\alpha_{i_{k}})\Big) \\*
    \notag &\qquad + \Phi_{\posetA}\big(\phi^{*}(\alpha_{i_{1}}) * \dots * \phi^{*}(\alpha_{i_{k-1}})\big) \circ \Phi_{\posetA}\big(\phi^{*}(\alpha_{i_{k}})\big) \\*
    \notag &\qquad - \Phi_{\posetA}\big(\phi^{*}(\alpha_{i_{1}}) * \dots * \phi^{*}(\alpha_{i_{k-1}}) * \phi^{*}(\alpha_{i_{k}})\big) \\*
    \notag &\qquad - \tildeH\,d(\alpha_{i_{1}} \cdots \alpha_{i_{k-1}}) \circ \Phi_{\posetA}(\phi^{*}(\alpha_{i_{k}})) \\*
    \notag &\qquad + \Phi_{\posetA}(\phi^{*}(\alpha_{i_{1}}))\circ \dots \circ \Phi_{\posetA}(\phi^{*}(\alpha_{i_{k-1}})) \circ \Phi_{\posetA}(\phi^{*}(\alpha_{i_{k}})) \\*
    \notag &\qquad - \Phi_{\posetA}\big(\phi^{*}(\alpha_{i_{1}})*\dots*\phi^{*}(\alpha_{i_{k-1}})\big) \circ \Phi_{\posetA}(\phi^{*}(\alpha_{i_{k}})) \\*
    \notag &\qquad + (-1)^{k}\, \tildeH(\alpha_{i_{1}}\cdots \alpha_{i_{k-1}}) \circ d\,\Phi_{\posetA}(\phi^{*}(\alpha_{i_{k}})) \\
    \notag &= - H\,d\Big(\phi^{*}(\alpha_{i_{1}}) * \dots * \phi^{*}(\alpha_{i_{k-1}})\otimes \phi^{*}(\alpha_{i_{k}})\Big) \\*
    \notag &\qquad - \tildeH\,d(\alpha_{i_{1}} \cdots \alpha_{i_{k-1}}) \circ \Phi_{\posetA}(\phi^{*}(\alpha_{i_{k}})) \\*
    \notag &\qquad - (-1)^{k-1}\, \tildeH(\alpha_{i_{1}}\cdots \alpha_{i_{k-1}}) \circ \Phi_{\posetA}(d\,\phi^{*}(\alpha_{i_{k}})) \\*
    \notag &\qquad + \Phi_{\posetA}(\phi^{*}(\alpha_{i_{1}}))\circ \dots \circ \Phi_{\posetA}(\phi^{*}(\alpha_{i_{k}})) \\*
    \notag &\qquad - \Phi_{\posetA}\big(\phi^{*}(\alpha_{i_{1}}) * \dots * \phi^{*}(\alpha_{i_{k}})\big).
  \end{align}
  Summing up the first three terms of the last expression,
  we get \(-\tildeH\,d(\alpha_{i_{1}}\cdots\alpha_{i_{k}})\),
  which proves the claim.

  \Cref{thm:Thetah-PhiSigma} finally implies that the map
  \begin{align}
    \hatH\colon \Kl_{\posetB} &\to \BB_{\cA}\bigl(\kk, C^{*}(B\LA),\kk[\posetA]\bigr), \\
    \notag \alpha\otimes f &\mapsto \tildeH(\alpha\otimes f)
    + \sum_{i>j}\tildeH\big(\iota(x_i)\,\iota(x_j)\,\alpha\otimes f\big)\,\hatq_{ij}.
  \end{align}
  is a homotopy from~\(\Phi_{\posetA}\,\hatKtw\) to~\(\Theta_{h^{*}} \,\Phi_{\posetB}\).
\end{proof}

The toric morphism~\((A,\nu)\colon (\NA,\posetA)\to(\NB,\posetB)\) defining \(\phi\)
naturally gives maps \(\phi^{*}\colon H^{*}(\LB)\to H^{*}(\LA)\)
and \(\hatphi^{*}\colon\kk[\VB]\to\kk[\VA]\), hence also a chain map
\begin{equation}
  \label{eq:def-phi-Kl}
  \Kl_{\posetB}\to\Kl_{\posetA},
  \qquad
  \alpha\otimes f \mapsto \phi^{*}(\alpha)\otimes\hatphi^{*}(f).
\end{equation}
We denote the induced map in cohomology by
\begin{equation}
  \Tor(\phi)\colon \Tor_{H^{*}(B\LB)}(\kk,\kk[\posetB]) \to \Tor_{H^{*}(B\LA)}(\kk,\kk[\posetA]).
\end{equation}
It is multiplicative with respect to the canonical products on the \(\Tor\)~terms.
Moreover, we write the map induced by~\(\hatKtw\) as
\begin{equation}
  \hatTor{\phi}\colon \Tor_{H^{*}(B\LB)}(\kk,\kk[\posetB]) \to \Tor_{H^{*}(B\LA)}(\kk,\kk[\posetA]).
\end{equation}
In contrast to~\(\Tor(\phi)\), this map does not only depend on~\((A,\nu)\),
but also on the twisting elements~\(\hatq_{ij}\in\kk[\posetA]\) as given by \Cref{thm:hatq}
and on the twisted \(*\)-product on~\(\Kl_{\posetA}\),
which in turn depends on the elements~\(q_{ij}\in\kk[\posetA]\) defined in~\eqref{eq:def-qij}.
Ultimately, \(\hatTor{\phi}\) depends on~\((A,\nu)\), on the bases chosen for~\(N\) and~\(N'\)
and on the representatives~\(c\) and~\(c'\).
The examples in the following section will show that \(\Tor{\phi}\) and~\(\hatTor{\phi}\) differ in general.

Summing up the discussion of this section, we obtain our main technical result.

\begin{theorem}
  \label{thm:main}
  In cohomology the diagram~\eqref{diag:main} induces the commutative diagram
  \begin{equation*}
    \begin{tikzcd}
      H^{*}(\parq{\posetB}{\KK}) \arrow{d}[left]{\Psi_{\posetB,\cB}} \arrow{r}{\phi^{*}} & H^{*}(\parq{\posetA}{\KK'}) \arrow{d}{\Psi_{\posetA,\cA}} \\
      \Tor_{H^{*}(B\LB)}(\kk,\kk[\posetB]) \arrow{r}{\hatTor{\phi}} & \Tor_{H^{*}(B\LA)}(\kk,\kk[\posetA]) \mathrlap{.}
    \end{tikzcd}
  \end{equation*}
\end{theorem}

\section{Applications}
\label{sec:application}

We continue to consider a map
\begin{equation}
  \phi=\phi_{(A,\nu)}\colon \parqdef{\posetA}{\KKA}\to\parqdef{\posetB}{\KKB}
\end{equation}
induced by a toric morphism of simplicial posets~\((A,\nu)\colon (\NA,\posetA)\to(\NB,\posetB)\).
We use the same notation as in the preceding section.
Recall from \Cref{thm:toric-top} that in case~\(\Sigma\) and~\(\Sigma'\) are rational fans
we have a commutative diagram
\begin{equation}
  \begin{tikzcd}
    \XX_{\SigmaA} \arrow{d} \arrow{r}{\phi^{*}} & \XX_{\SigmaB} \arrow{d} \\
    \XXX_{\SigmaA} \arrow{r}{\phi^{*}} & \XXX_{\SigmaB}
  \end{tikzcd}
\end{equation}
whose vertical arrows are equivariant homotopy equivalences.
It therefore suffices to prove the analogues of \Cref{thm:intro:nat-2,thm:intro:nat} for partial quotients.
This is done in \Cref{sec:nat-general,sec:nat-2}, respectively.
The same remark applies to \Cref{sec:compare}
where we compare the twisted and the untwisted product in case \(2\) is invertible in~\(\kk\).
In \Cref{sec:macs} we make some additional comments regarding moment-angle complexes.

\subsection{Naturality in general}
\label{sec:nat-general}

The goal of this section is to prove \Cref{thm:intro:nat},
which applies to any coefficient ring~\(\kk\). In the light of \Cref{thm:main},
the main work left to do is construct explicit chains~\(b_{\vA}\in C_{2}(T_{\nu(\vA)})\)
such that the condition~\eqref{eq:def-bv} holds for all vertices~\(\vA\) of~\(\posetA\).
In a second step we evaluate the formula for~\(\hatq_{ij}\) obtained in \Cref{thm:hatq}.

Unless stated otherwise, we choose the canonical representative~\(c=[1]\in C_{1}(S^{1})\)
of the generator~\(1\in H_{1}(S^{1})=\Z\) for the rest of this paper. It leads to the multiplicative isomorphism
\begin{equation}
  \Psi_{\posetB} = \Psi_{\posetB,c}\colon H^{*}(\parq{\posetB}{\KK}) \to \Tor_{H^{*}(BL)}\bigl(\kk,\kk[\posetB]\bigr)
\end{equation}
where the multiplication on the torsion product is induced by the \(*\)-product~\eqref{eq:twisted-prod} on~\(\Kl_{\posetB}\).
The twisting terms~\(q_{ij}\) from~\eqref{eq:def-qij} evaluate to~\eqref{eq:def-qij-general} in this case.
All this is \cite[Thm.~7.2]{Franz:torprod}, which we have recalled in \Cref{thm:iso-mult-general}.

Recall from \Cref{sec:morph} that a toric morphism~\((A, \nu)\colon (\NA,\posetA)\to (\NB,\posetB)\) of simplicial posets (or regular fans)
lifts to a toric morphism~\({(\hatA,\nu)}\colon (\hatNA, \posetA) \to (\hatNB, \posetB)\)
where \(\hatA=(\hatawv{\vB}{\vA})\in\Z^{\VB\times \VA}\).
Moreover, if \(\vA\) is a vertex of~\(\posetA\) (and not a ghost vertex),
then \(\hatawv{\vB}{\vA}\ge0\) if \(\vA\in\nu(\vB)\) and equal to~\(0\) otherwise.

We take a vertex~\(\vA\in\posetA\) and set \(\tau=\nu(\vA)\).
Remember that we are working with simplicial tori, which are bar constructions of lattices.
The simplicial torus~\(T_{\tau}\) is the bar construction of the sublattice of~\(\hatNB\)
spanned by the basis vectors~\(e_{\vB}\) with~\(\vB\in\tau\).

For each~\(\vB\in\tau\) we define the \(2\)-chains
\begin{align}
   C_{2}(\TB_{\vB}) \ni
  b_{\vA,\vB} &=
  \begin{cases}
    [e_{\vB},e_{\vB}]+\cdots+[(\hatawv{\vB}{\vA}-1)\,e_{\vB},e_{\vB}] & \text{if~\(\hatawv{\vB}{\vA}\geq 2\),} \\
    0 & \text{if~\(\hatawv{\vB}{\vA}=0\) or~\(1\)}
  \end{cases} \\
  \shortintertext{and}
  \label{eq:bv}
  C_{2}(\TB_{\tau}) \ni
  b_{\vA} &= -\sum_{\vB\in\tau} b_{\vA,\vB}
    -\sum_{\wB\in\tau} \Biggl[ \sum_{\vB<\wB} \hatawv{\vB}{\vA} e_{\vB}, \hatawv{\wB}{\vA} e_{\wB} \Biggr],
\end{align}
where we compare the vertices in~\(\VB\) according to the chosen order. Then
\begin{align}
  d\,b_{\vA,\vB} &= \hatawv{\vB}{\vA}[e_{\vB}]-[\hatawv{\vB}{\vA} e_{\vB}], \\
  \shortintertext{and therefore also}
  d\,b_{\vA} &= \Biggl[ \sum_{\vB\in \VB} \hatawv{\vB}{\vA} e_{\vB} \Biggr] - \sum_{\vB\in \VB} \hatawv{\vB}{\vA}\,[e_{\vB}] 
    = \check{c}_{\vA} - \hat{c}_{\vA},
\end{align}
where \(\check{c}_{\vA} \) and~\(\hat{c}_{\vA}\) are defined as in~\eqref{eq:def-bv}.

\Cref{thm:hatq} now gives
\begin{equation}
  \label{eq:hatq-general-1}
  \hatq_{ij} = - \sum_{\vA\in\posetA}\Biggl(\,
  \sum_{\vB\in\nu(\vA)} \frac{\hatawv{\vB}{\vA}(\hatawv{\vB}{\vA}-1)}{2}\, \xvi{\vB}{i}\,\xvi{\vB}{j}
  + \sum_{\substack{\vB,\wB\in\nu(\vA)\\\vB<\wB}} \hatawv{\vB}{\vA}\,\hatawv{\wB}{\vA}\,\xvi{\vB}{i}\,\xvi{\wB}{j}
  \Biggr)\,t_{\vA}
\end{equation}
for~\(1\leq j<i\leq \nB\),
where \((\xvi{\vB}{i})\) is the characteristic matrix for~\(\parq{\posetB}{\KK}\)
and the first sum extends over all non-ghost vertices~\(\vA\).

As mentioned above, \(\hatawv{\vB}{\vA}\) vanishes for~\(\vA\in\posetA\) and~\(\vB\notin\nu(\vA)\),
and so does \(t_{\vA}\) for each ghost vertex~\(\vA\). We can therefore rewrite formula~\eqref{eq:hatq-general-1}
in the form
\begin{equation}
  \label{eq:hatq-general}
  \hatq_{ij} = - \sum_{\vA\in\VA}\Biggl(\,
  \sum_{\vB\in \VB} \frac{\hatawv{\vB}{\vA}(\hatawv{\vB}{\vA}-1)}{2}\, \xvi{\vB}{i}\,\xvi{\vB}{j}
  + \sum_{\substack{\vB,\wB\in \VB\\\vB<\wB}} \hatawv{\vB}{\vA}\,\hatawv{\wB}{\vA}\,\xvi{\vB}{i}\,\xvi{\wB}{j}
  \Biggr)\,t_{\vA}.
\end{equation}

\begin{remark}
  \label{rem:hatqij-unique}
  We pointed out in \Cref{rem:hatA-unique} that the coefficient~\(\hatawv{\vB}{\vA}\) of the lift~\(\hatA\) of~\(A\)
  is uniquely determined by~\((A,\nu)\) if \(\vA\) is not a ghost vertex.
  Hence formula~\eqref{eq:hatq-general-1} shows that the twisting terms~\(\hatq_{ij}\) are independent
  of the chosen lifting, and so are the map~\(\hatKtw\) defined in~\eqref{eq:def-hatKtw}
  and the induced map~\(\hatTor{\phi}\) in cohomology.
\end{remark}

\begin{theorem}
  \label{thm:nat-general}
  The following diagram commutes.
  \begin{equation*}
    \begin{tikzcd}
      H^{*}(\parq{\posetB}{\KK}) \arrow{d}[left]{\Psi_{\posetB}} \arrow{r}{\phi^{*}} & H^{*}(\parq{\posetA}{\KK'}) \arrow{d}{\Psi_{\posetA}} \\
      \Tor_{H^{*}(B\LB)}(\kk,\kk[\posetB]) \arrow{r}{\hatTor{\phi}} & \Tor_{H^{*}(B\LA)}(\kk,\kk[\posetA]) 
    \end{tikzcd}
  \end{equation*}
  Here the twisting terms~\(\hatq_{ij}\) implicit in the definition of~\(\hatTor{\phi}\)
  are given by~\eqref{eq:hatq-general}.
\end{theorem}

\begin{proof}
  The statement is a consequence of \Cref{thm:main} and \Cref{thm:hatq}
  with a substitution of the elements~\(c_{v}\in C_{1}(T)\),~\(c_{v'}\in C_{1}(T')\) as well as \(b_{\vA}\) from~\eqref{eq:bv}.
\end{proof}

\begin{remark}
  We continue the discussion started in \Cref{rem:equiv-formal-prod}.
  Given any toric morphism~\(\phi\colon \parq{\posetA}{\KK'}\to \parq{\posetB}{\KK}\),
  we see from the definition~\eqref{eq:def-hatKtw} that the map~\(\hatKtw\)
  sends any~\(1\otimes f\in\Kl_{\posetB}\) to~\(1\otimes\hatphi^{*}(f)\in\Kl_{\posetA}\).
  If \(\parq{\posetB}{\KK}\) is equivariantly formal (which happens for instance
  in the case of smooth compact toric varieties), this implies
  that the induced map~\(\hatTor{\phi}\) in cohomology is untwisted in the sense that it coincides with~\(\Tor(\phi)\).
  This is not surprising because the equivariant formality of~\(\parq{\posetB}{\KK}\) entails that
  the map~\(H^{*}(\parq{\posetB}{\KK})\to H^{*}(\parq{\posetA}{\KK'})\) is completely determined by its equivariant counterpart.
\end{remark}

\begin{example}
  \label{ex:nontrivial-qij-basis}
  Consider the fan \(\Sigma'\) in \(\R^3\) whose maximal cones are rays spanned
  by the vectors~\(v'=e_{3}=[0,0,1]\) and \(w'=-v'=[0,0,-1]\).
  Then \(\XXX_{\Sigma'}\cong\XXX_\Sigma\), where \(\XXX_\Sigma\) is as in \Cref{ex:nontrivial-qij-mult}.
  Using the canonical basis for~\(N'=\Z^3\), 
  we get the following generators for the cohomology \(H^{*}(\XXX_{\Sigma'})\) in the specified bidegrees:
  \begin{align}
    (-1,2) &\colon a'_{1} = [\alpha'_{1}], \quad a'_{2} = [\alpha'_{2}], \\
    (-2,4) &\colon b' = [\alpha'_{1}\,\alpha'_{2}], \\
    (0,2) &\colon c' = [t_{v'}] = [t_{w'}].
  \end{align}
  The only non-zero twisting term is \(q'_{33}=t_{w'}\), which entails that the multiplication is untwisted in cohomology.

  The map~\(A\colon N'\to N\), \(e_{1}\mapsto e_{1}\),~\(e_{2}\mapsto e_{2}\),~\(e_{3}\mapsto v\)
  determines an isomorphism of toric varieties
  \( 
    \phi \colon \XXX_{\Sigma'}\to \XXX_{\Sigma}
  \). 
  Its lift~\(\hatA\colon \hatN'\to \hatN\) is the identity map on~\(\Z^{4}\). We have
  \begin{gather}
    \phi^{*}(\alpha_{1}) = \alpha'_{1}+\alpha'_{3},
    \qquad
    \phi^{*}(\alpha_{2}) = \alpha'_{2}+\alpha'_{3},
    \qquad
    \phi^{*}(\alpha_{3}) = \alpha'_{3}, \\
    \phi^{*}(\beta) = \beta'
    \qquad
    \hatphi^{*}(c) = c',
  \end{gather}
  so that for the untwisted map~\(\Tor(\phi)\) we get
  \begin{gather}
    \Tor(\phi)\colon
    a_{1} \mapsto a'_{1},
    \quad
    a_{2} \mapsto a'_{2}, 
    \quad
    b \mapsto b',
    \quad
    c \mapsto c'.
  \end{gather}
  This implies that the two cohomology classes
  \begin{align}
    \Tor(\phi)(a_{1} * a_{2}) &= \Tor(\phi)(b-c) = b'-c', \\
    \Tor(\phi)(a_{1})*\Tor(\phi)(a_{2}) &= a'_{1} * a'_{2} = b'
  \end{align}
  differ. Thus \(\Tor(\phi)\) is not multiplicative.

  We now look at the twisted map~\(\hatTor{\phi}\). Since \(\hatA\) is the identity map,
  formula~\eqref{eq:hatq-general} yields trivial twisting terms~\(\hatq_{ij}\), hence \(\hatKtw=\Ktw\).
  We therefore have
  \begin{align}
    \MoveEqLeft{\hatTor{\phi}(a_{1}) = a'_{1}, \qquad \hatTor{\phi}(a_{2}) = a'_{2}, \qquad\hatTor{\phi}(c) = c',} \\
    \hatTor{\phi}(b) &= \bigl[\hatKtw(\beta)\bigr] = \bigl[\Ktw(\alpha_{1}\,\alpha_{2} + \alpha_{2}\,\alpha_{3} - \alpha_{1}\,\alpha_{3})\bigr] \\
    \notag &= [(\alpha'_{1}+\alpha'_{3})*(\alpha'_{2}+\alpha'_{3}) + (\alpha'_{2}+\alpha'_{3})*\alpha'_{3} - (\alpha'_{1}+\alpha'_{3})*\alpha'_{3}]\\
    \notag &= [\alpha'_{1}\,\alpha'_{2} + q'_{33}] = b'+c',
  \end{align}
  which gives
  \begin{equation}
    \hatTor{\phi}(a_{1}*a_{2}) = \hatTor{\phi}(b-c) = b' = \hatTor{\phi}(a_{1})*\hatTor{\phi}(a_{2}),
  \end{equation}
  as predicted by \Cref{thm:nat-general}.
\end{example}

We further illustrate \Cref{thm:nat-general} by the diagonal map and power maps.
Recall that the \newterm{join}~\(\Sigma_1*\Sigma_2\) of the simplicial posets~\(\Sigma_1\) and~\(\Sigma_2\)
is a simplicial poset consisting of the pairs~\((\sigma, \tau)\) with~\(\sigma\in \Sigma_1\) and~\(\tau\in \Sigma_2\)
(which may be empty simplices).

\begin{example}
  We consider the diagonal map
  \begin{equation}
    \parq{\posetB}{\KK}\to\parq{\posetB}{\KK}\times\parq{\posetB}{\KK} = \bigparq{\posetB*\posetB}{\KK\times \KK}
  \end{equation}
  of a partial quotient~\(\parq{\posetB}{\KK}=\ZZ_{\posetB}/\KK\).
  The vertex set of the join~\(\posetB*\posetB\) consists of two disjoint copies~\(V_{1}\) and~\(V_{2}\) of the vertex set~\(V\) for~\(\posetB\).
  We write the two vertices corresponding to~\(v\in V\) as~\(v_{1}\in V_{1}\) and~\(v_{2}\in V_{2}\).
  We extend the ordering of~\(V\) to~\(V_{1}\sqcup V_{2}\) by defining all vertices in~\(V_{1}\)
  to be smaller than all vertices in~\(V_{2}\). The characteristic matrix for~\(\bigparq{\posetB*\posetB}{\KK\times \KK}\)
  is determined by
  \begin{equation}
    \xv{v_{1}} = (v,0) \in N\oplus N,
    \qquad\text{and}\qquad
    \xv{v_{2}} = (0,v) \in N\oplus N
  \end{equation}
  for~\(v\in V\). The chosen coordinates for~\(N\) give coordinates for~\(N\oplus N\)
  by first taking the \(n\)~coordinates of the first summand, followed by the \(n\)~coordinates
  of the second summand.
  
  The diagonal map~\(A\colon N\to N\oplus N\) lifts to~\(\hatA\colon\hatN\to\hatN\oplus\hatN\) with
  \begin{equation}
    \hatawv{w}{v} = \begin{cases}
      1 & \text{if \(w=v_{1}\) or~\(w=v_{2}\),} \\
      0 & \text{otherwise} \\
    \end{cases}
  \end{equation}
  for~\(v\in V\) and~\(w\in V_{1}\sqcup V_{2}\).
  This implies that the first term inside the bracket in the formula~\eqref{eq:hatq-general}
  for~\(\hatq_{ij}\) vanishes. In fact, the second term vanishes, too:
  For a given~\(v\in V\) only the coefficients~\(\hatawv{v_{1}}{v}\) and~\(\hatawv{v_{2}}{v}\)
  are non-zero. For~\(v_{1}\) only the first half of the coordinates can be non-zero,
  and for~\(v_{2}\) only the second half. Hence \(\xvi{v_{1}}{i}\ne0\)
  implies \(\xvi{v_{2}}{j}=0\) for~\(j<i\). All coefficients~\(\hatq_{ij}\) thus vanish.

  Because the way we have chosen the coordinates on~\(N\oplus N\), this means that the map
  \begin{equation}
    \hatKtw = \Ktw \colon \Kl_{\posetB*\posetB} = \Kl_{\posetB}\otimes\Kl_{\posetB} \to \Kl_{\posetB}
  \end{equation}
  is in fact the twisted \(*\)-product.
  Since the diagonal map induces the cup product in cohomology,
  one can deduce from this the description of the cup product given in \Cref{thm:iso-mult-general}.
  As the examples given in~\cite[Ex.~1.1, Sec.~9]{Franz:torprod} show, the cup product is twisted in general,
  which again illustrates that \(\hatTor{\phi}\) may differ from the canonical map between the torsion products.
\end{example}

\begin{example}
  Let \(r\in\N\). Given a simplicial poset~\(\posetB\) and a closed subgroup~\(\KK\subset\TT\) acting freely on~\(\ZZ_{\posetB}\)
  (or instead a regular fan~\(\Sigma\)),
  consider the \(r\)-th power map~\(\phi\) on the partial quotient~\(\parq{\posetB}{\KK}\) (or the smooth toric variety~\(\XXX_{\Sigma}\)).
  This is a toric morphism with \(A=r\,\Id_{N}\) and \(\nu=\Id_{\posetB}\).
  In this case the map~\(\Ktw\) is given by
  \begin{equation}\label{map: phi power map}
    \Ktw(\alpha\otimes f) = r^{k+l}\,\alpha\otimes f
  \end{equation}
  for~\(\alpha\in H^{k}(L)\) and~\(f\in\kk[\posetB]\) of degree~\(2l\), and we also have \(\hatA=r\,\Id_{\hatN}\), that is,
  \begin{equation}
    \hatawv{\vB}{\vA} = \begin{cases}
      r & \text{if \(\vB=\vA\),} \\
      0 & \text{otherwise.}
    \end{cases}
  \end{equation}
  Thus, the formula~\eqref{eq:hatq-general} reduces to
  \begin{equation}
    \hatq_{ij} = - \frac{r\,(r-1)}{2}\,q_{ij}.
  \end{equation}
  We therefore get
  \begin{equation}
    \label{map_hatphi_powermap}
    \hatKtw (\alpha\otimes f) = r^{k+l}\,\alpha\otimes f
    - \frac{r^{k+l-1}(r-1)}{2} \sum_{i> j} \iota(x_{i})\,\iota(x_{j})\,\alpha \otimes f\,q_{ij}.
  \end{equation}
  Note that the sum on the right vanishes for~\(k<2\). For~\(k\ge2\) we have \(k+l-1\ge1\),
  which implies that the factor in front of the sum is an integer.
\end{example}

\begin{example}
  \label{ex:nontrivial-qij-powermap}
  We consider the \(r\)-th power map~\(\phi\colon \XXX_\Sigma\to \XXX_\Sigma\) for the toric variety given in \Cref{ex:nontrivial-qij-mult}.
  By formula~\eqref{map_hatphi_powermap} we have in cohomology
  \begin{gather}
    \hatTor{\phi}(a_{1}) = r\,a_{1}, \qquad \hatTor{\phi}(a_{2}) = r\,a_{2}, \qquad \hatTor{\phi}(c) = r\,c, \\
    \hatTor{\phi}(b) = r^{2}\,[\beta] - \frac{r(r-1)}{2}\,\bigl[q_{21}+q_{32}-q_{31}\bigr] = r^{2}\,b - r(r-1)\,c.
  \end{gather}
  To illustrate that \(\hatTor{\phi}\) is indeed multiplicative, we compute
  \begin{multline}
    \qquad \hatTor{\phi}(a_{1}*a_{2}) = \hatTor{\phi}(b-c) = r^{2}\,b - r(r-1)\,c - r\,c \\
    = r^{2}\,(b-c) = (r\,a_{1})*(r\,a_{2}) =\hatTor{\phi}(a_{1})*\hatTor{\phi}(a_{2}). \qquad
  \end{multline}
\end{example}

\subsection{Moment-angle complexes and Cox constructions}
\label{sec:macs}

Every moment-angle complex~\(\ZZ_{\posetB}\) is itself a partial quotient (with~\(\KK=1\)).
In the same vein, every Cox construction~\(\ZZZ_{\Sigma}\) (or complement of a complex coordinate subspace arrangement)
is a toric variety. As shown by Franz~\cite[Thm.~1.3]{Franz:2003a} and Baskakov--Buchstaber--Panov~\cite{BaskakovBuchstaberPanov:2004},
the multiplication is untwisted in this case. This is confirmed by the fact that all~\(q_{ij}\)'s vanish if
the \(\xv{v}\)'s are part of the canonical basis~\((e_{v})\) for~\(N=\hatN\), compare~\cite[Rem.~7.4]{Franz:torprod}
and also \Cref{ex:nontrivial-qij-mult}.
In this section we make several additional observations regarding moment-angle complexes and Cox constructions.

\begin{proposition}
  \label{thm:tw-vanish-mac}
  Let \((A,\nu)\colon (\hatNA,\posetA)\to(\hatNB,\posetB)\) be a toric morphism,
  inducing a map~\(\phi\colon \ZZ_{\posetA}\to\ZZ_{\posetB}\). Then
  \begin{equation*}
    \hatTor{\phi} = \Tor(\phi)\colon \Tor_{H^{*}(B\TB)}\bigl(\kk,\kk[\posetB]\bigr) \to \Tor_{H^{*}(B\TA)}\bigl(\kk,\kk[\posetA]\bigr).
  \end{equation*}
\end{proposition}

Together with \Cref{thm:nat-general} this means that
the map~\(\Psi_{\posetB}\colon H^{*}(\ZZ_{\posetB})\to \Tor_{H^{*}(BT)}(\kk,\kk[\posetB])\)
is natural with respect to toric morphisms between moment-angle complexes and the canonical maps
between the associated \(\Tor\)~terms.

\begin{proof}
  We start by looking at formula~\eqref{eq:hatq-general} for~\(\hatq_{ij}\).
  Since the vectors~\(\xv{\vB}=e_{\vB}\) are part of the canonical basis for~\(\NB=\hatNB\)
  (with the same indexing set), we have
  \begin{equation}
    \xvi{\vB}{i} = \begin{cases}
      1 & \text{if \(i=\vB\),} \\
      0 & \text{otherwise.}
    \end{cases}
  \end{equation}
  This implies that the first term inside the brackets in~\eqref{eq:hatq-general} vanishes,
  and in fact also the second because we have \(j<\wB\) if \(i=\vB\). Hence all~\(\hatq_{ij}\)'s vanish.
  
  Because the product in~\(\Kl_{\posetA}\) is untwisted, this implies that \(\hatKtw=\Ktw\)
  is the canonical map~\eqref{eq:def-phi-Kl} induced by~\((A,\nu)\). This proves the claim.
\end{proof}

\begin{proposition}
  \label{thm:tw-vanish-proj}
  Let \(\kappa\colon\ZZ_{\posetB}\to\parq{\posetB}{\KK}\) be the canonical projection map of a partial quotient. Then
  \begin{equation*}
    \hatTor{\kappa} = \Tor(\kappa)\colon \Tor_{H^{*}(BL)}\bigl(\kk,\kk[\posetB]\bigr) \to \Tor_{H^{*}(BT)}\bigl(\kk,\kk[\posetB]\bigr),
  \end{equation*}
  where the canonical basis for~\(N'=\hatN\) is used when computing \(\hatTor{\kappa}\).
\end{proposition}

\begin{proof}
  Since the domain of~\(\kappa\) is a moment-angle complex,
  this is analogous to the proof of \Cref{thm:tw-vanish-mac}.
\end{proof}

Let us define the ideal
\begin{equation}
  \label{eq:def-I-Sigma}
  I_{\posetB} = \ker \Tor(\kappa) = \ker \hatTor{\kappa} \lhd \Tor_{H^{*}(BL)}\bigl(\kk,\kk[\posetB]\bigr).
\end{equation}

\begin{corollary}
  Let \((A,\nu)\colon (\NA,\posetA)\to(\NB,\posetB)\) be a toric morphism,
  inducing a map~\(\phi\colon\parq{\posetA}{\KK'}\to\parq{\posetB}{\KK}\) between partial quotients.
  Modulo~\(I_{\posetA}\), we have a congruence
  \begin{equation*}
    \hatTor{\phi} \equiv \Tor(\phi)\colon \Tor_{H^{*}(B\LB)}\bigl(\kk,\kk[\posetB]\bigr) \to \Tor_{H^{*}(B\LA)}\bigl(\kk,\kk[\posetA]\bigr).
  \end{equation*}
\end{corollary}

\begin{proof}
  We consider the following diagram.
  \begin{equation}
    \begin{tikzcd}[column sep=small]
      H^{*}(\ZZ_{\posetB}) \arrow[dd,"\hatphi^{*}"'] \arrow[rrr,"\Psi_{\posetB}"] & & & \Tor_{H^{*}(B\TB)}\mathrlap{(\kk,\kk[\posetB])} \arrow[dd,"\Tor(\hatphi)" near end] \\
      & H^{*}(\parq{\posetB}{\KK}) \arrow[ul,"\kappaB^{*}"] \arrow[rrr,crossing over,"\Psi_{\posetB}" near start] & & & \Tor_{H^{*}(B\LB)}\mathrlap{(\kk,\kk[\posetB])} \arrow[dd,"\Upsilon"] \arrow[ul,"\Tor(\kappaB)"' near start] \\
      H^{*}(\ZZ_{\posetA}) \arrow[rrr,"\Psi_{\posetA}"' near end] & & & \Tor_{H^{*}(B\TA)}\mathrlap{(\kk,\kk[\posetA])} \\
      & H^{*}(\parq{\posetA}{\KK'}) \arrow[from=uu,crossing over,"\phi^{*}"' near start] \arrow[ul,"(\kappaA)^{*}"] \arrow[rrr,"\Psi_{\posetA}"'] & & & \Tor_{H^{*}(B\LA)}\mathrlap{(\kk,\kk[\posetA])} \arrow[ul,"\Tor(\kappaA)"]
    \end{tikzcd}
  \end{equation}
  The left square commutes by naturality, and so does the right one
  if one takes \(\Upsilon=\Tor(\phi)\) as the rightmost vertical map.
  The back commutes by \Cref{thm:tw-vanish-mac} and the top and bottom squares by \Cref{thm:tw-vanish-proj},
  each time combined with \Cref{thm:nat-general}. The latter result also guarantees
  that the front commutes for~\(\Upsilon=\hatTor{\phi}\).
  Altogether this implies that \(\hatTor{\phi}\) and~\(\Tor(\phi)\) agree modulo~\(I_{\posetA}\).
\end{proof}

\begin{remark}
  \label{rem:I-Sigma-large}
  The ideal~\(I_{\posetB}\) can be quite large: It is not difficult to see
  that the submodule~\(1\otimes\kk[\posetB]\) of the Koszul complex for~\(\ZZ_{\posetB}\) consists of coboundaries.
  Hence \(I_{\posetB}\) contains the image of~\(\kk[\posetB]\) in~\(\Tor_{H^{*}(BL)}(\kk,\kk[\posetB])\),
  that is, the restrictions of all equivariant cohomology classes of~\(\parq{\posetB}{\KK}\).
  If \(\parq{\posetB}{\KK}\) is equivariantly formal, then this implies that \(I_{\posetB}\) contains all cohomology classes of positive degree.
\end{remark}

\subsection{Naturality if \texorpdfstring{\(2\)}{2} is invertible}
\label{sec:nat-2}

In this section we prove \Cref{thm:intro:nat-2} and its generalization to partial quotients.

If \(2\) is invertible in~\(\kk\), then we can choose the representative
\begin{equation}
  \label{eq:def-tildec}
  \tilde{c} = \smallhalf\,[1] - \smallhalf\,[-1] = \smallhalf (c - \iota_{*}\,c) \in C_{1}(S^{1}),
\end{equation}
where \(c=[1]\) is the representative used in \Cref{sec:nat-general}
and \(\iota\colon T\to T\) the group inversion. This leads to the isomorphism
\begin{equation}
  \tilde{\Psi}_{\posetB} = \Psi_{\posetB,\tilde{c}} \colon H^{*}(\parq{\posetB}{\KK}) \to \Tor_{H^{*}(BL)}\bigl(\kk,\kk[\posetB]\bigr)
\end{equation}
constructed in \cite[Sec.~8]{Franz:torprod} and featuring in \Cref{thm:iso-mult-2}, and analogously we obtain~\(\tilde{\Psi}_{\posetA}\).

For any~\(\vA\in\VA\), let \(b_{\vA}\in C_{2}(\TB)\) be a chain as in~\eqref{eq:def-bv}
based on the representative~\(c\) both for~\(\DJ_{\posetB}\) and~\(\DJ_{\posetA}\). For example, we can take
the elements chosen in \Cref{sec:nat-general}.
We obtain new elements~\(\tilde{b}_{\vA}\) in the same way as in~\eqref{eq:def-tildec},
\begin{equation}
  \tilde{b}_{\vA} = \smallhalf (b_{\vA} - \iota_{*}\,b_{\vA}) \in C_{2}(\TB)
\end{equation}
for~\(\vA\in\VA\). If we write \(\ell(\vA)\) for the number of \(2\)-simplices~\(\bis{\vA}{s}\)
appearing in~\(b_{\vA}\), then we can express each
\begin{equation}
  \tilde{b}_{\vA} = \sum_{s=1}^{2\ell(\vA)} \tilde{b}_{\vA s}
\end{equation}
as a linear combination of \(2\,\ell(\vA)\) (not necessarily distinct) \(2\)-simplices where
\begin{equation}
  \tildelambdavs{\vA}{s} = \smallhalf\,\lambdavs{\vA}{s},
  \quad
  \tilde{b}_{\vA s} = \bis{\vA}{s},
  \quad
  \tildelambdavs{\vA}{s+\ell(\vA)} = -\smallhalf\,\lambdavs{\vA}{s},
  \quad
  \tilde{b}_{\vA,s+\ell(\vA)} = \iota_{*}\,\bis{\vA}{s}
\end{equation}
for~\(1\le s\le\ell(\vA)\).

We have \(\iota_{*}\,[\ww^{1},\ww^{2}]=[-\ww^{1},-\ww^{2}]\) for any \(2\)-simplex~\([\ww^{1},\ww^{2}]\) in~\(S^{1}=B\Z\).
For the coordinates~\(\tildewwonevsv{\vA}{s}{\vB}\) and~\(\tildewwtwovsv{\vA}{s}{\vB}\)
of the \(2\)-simplices~\eqref{eq:def-wvj} this implies in our present context that
\begin{equation}
  \tildewwonevsv{\vA}{s+\ell(\vA)}{\vB} = - \tildewwonevsv{\vA}{s}{\vB}
  \qquad\text{and}\qquad
  \tildewwtwovsv{\vA}{s+\ell(\vA)}{\vB} = - \tildewwtwovsv{\vA}{s}{\vB}
\end{equation}
for all~\(v\in V\),~\(v'\in V'\) and~\(1\le s\le\ell(\vA)\). A look at \Cref{thm:hatq} now reveals
\begin{equation}
  \hatq_{ij} = 0
\end{equation}
for all~\(i\ge j\) because each term for~\(s\) cancels against the corresponding term for \(s+\ell(\vA).\)
The twisting terms~\(\tilde{q}_{ij}\) governing the \(*\)-product in~\(\Kl_{\posetA}\)
also vanish by \cite[Lemma~8.1]{Franz:torprod}. Hence
\begin{equation}
  \hatKtw(\alpha\otimes f) = \Ktw(\alpha\otimes f) = \phi^{*}(\alpha) \otimes \hatphi^{*}(f) \in \Kl_{\posetA}
\end{equation}
for any~\(\alpha\otimes f\in\Kl_{\posetB}\). \Cref{thm:main} therefore takes on the following form.

\begin{theorem}
  \label{thm:nat-2}
  Assume that \(2\) is invertible in~\(\kk\). The following diagram commutes.
  \begin{equation*}
    \begin{tikzcd}[column sep=huge]
      H^{*}(\parq{\posetB}{\KK}) \arrow{d}[left]{\tilde{\Psi}_{\posetB}} \arrow{r}{\phi^{*}} & H^{*}(\parq{\posetA}{\KK'}) \arrow{d}{\tilde{\Psi}_{\posetA}} \\
      \Tor_{H^{*}(B\LB)}\bigl(\kk,\kk[\posetB]\bigr) \arrow{r}{\Tor(\phi)} & \Tor_{H^{*}(B\LA)}\bigl(\kk,\kk[\posetA]\bigr)
    \end{tikzcd}
  \end{equation*}
\end{theorem}

\begin{remark}
  \Cref{thm:nat-2} can also be proven with the (more high-powered) techniques
  which are developed in~\cite{Franz:hgashc} and~\cite{Franz:homog} to show that
  for certain homogeneous spaces~\(G/K\) the isomorphism
  \begin{equation}
    H^{*}(G/K) \cong \Tor_{H^{*}(BG)}\bigl(\kk,H^{*}(BK)\bigr)
  \end{equation}
  is multiplicative and natural if \(2\) is invertible in~\(\kk\). In this approach one constructs an \Ainfty~map~%
  \( 
    \Lambda^{\posetB}\colon \kk[\posetB] \Rightarrow C^{*}(\DJ_{\posetB})
  \) 
  along with an \Ainfty~map~%
  \(
    \Lambda^{L}\colon H^{*}(BL) \Rightarrow C^{*}(BL)
  \). 
  The key fact is that the diagram
  \begin{equation}
    \begin{tikzcd}
      H^{*}(BL) \arrow{r}{\rho^{*}} \arrow[Rightarrow]{d}[left]{\Lambda^{L}} & \kk[\posetB] \arrow[Rightarrow]{d}{\Lambda^{\posetB}} \\
      C^{*}(BL) \arrow{r}{\rho^{*}} & C^{*}(DJ_{\posetB})
    \end{tikzcd}
  \end{equation}
  commutes up to an \Ainfty~homotopy that 
  behaves well with respect to the formality map~\(f_{\posetB,\tilde{c}}^{*}\colon C^{*}(DJ_{\posetB})\to\kk[\posetB]\),
  compare \cite[Cor.~10.9, Thm.~11.6]{Franz:homog}.\footnote{%
    This technical issue is crucial for proving the multiplicativity of the isomorphism.
    It is not addressed in the note~\cite{Panov:2015}, which is otherwise based on similar ideas.}
\end{remark}

\subsection{Comparing the two products on the Koszul complex}
\label{sec:compare}

Let \(\XX_{\Sigma}=\ZZ_{\posetB}/\KK\) be a partial quotient with quotient map~\(\kappa\).
For the moment, we allow a general coefficient ring~\(\kk\).
Recall from \Cref{thm:tw-vanish-proj} that \(\Tor(\kappa)=\hatTor{\kappa}\)
and that in~\eqref{eq:def-I-Sigma} we have defined the ideal~\(I_{\posetB}\lhd\Tor_{H^{*}(BL)}(\kk,\kk[\posetB])\)
as the kernel of this map.

\begin{proposition}
  For all~\(a\),~\(b\in\Tor_{H^{*}(BL)}(\kk,\kk[\posetB])\) we have
  \begin{equation*}
    a*b \equiv a\,b \pmod{I_{\posetB}}.
  \end{equation*}
\end{proposition}

In other words, the deformation terms for the product in the \(\Tor\)~term for~\(\parq{\posetB}{\KK}\) vanish
when they are pulled back to~\(\ZZ_{\posetB}\) via \(\Tor(\kappa)=\hatTor{\kappa}\).

\begin{proof}
  By \Cref{thm:nat-general} the map~\(\hatTor{\kappa}\) is multiplicative
  with respect to the twisted products on both \(\Tor\)~terms. On the other hand,
  \(\Tor(\kappa)\) is so with respect to the canonical products. As remarked
  at the beginning of \Cref{sec:macs}, the \(*\)-product in~\(\Tor_{H^{*}(BT)}(\kk,\kk[\posetB])\) is in fact untwisted.
  Since the maps just mentioned agree, the difference~\(a*b-a\,b\) must lie in their common kernel.
\end{proof}

If \(2\) is invertible in~\(\kk\), then it follows from \Cref{thm:iso-mult-2,thm:iso-mult-general}
that both the twisted product and the canonical product on the Koszul complex~\(\Kl_{\posetB}\)
induce a multiplicative isomorphism in cohomology to~\(H^{*}(\parq{\posetB}{\KK})\).
We finally describe an automorphism of~\(\Kl_{\posetB}\) that in cohomology is multiplicative with respect to these two products.

\begin{proposition}
  \label{thm:invert2}
  Assume that \(2\) is invertible in~\(\kk\). The \(\kk[\posetB]\)-linear automorphism
  \begin{equation*}
     \AA\colon \Kl_{\posetB} \to \Kl_{\posetB},
    \qquad
    \alpha\otimes f \mapsto \alpha\otimes f
    + \frac{1}{2}\,\sum_{i>j} \iota(x_i)\,\iota(x_j)\,\alpha \otimes f\,q_{ij}
  \end{equation*}
  induces a multiplicative isomorphism in cohomology that makes the diagram
  \begin{equation*}
    \begin{tikzcd}[column sep=small]
      & H^{*}(\parq{\posetB}{\KK}) \arrow[dl,"\Psi_{\posetB}"'] \arrow[dr,"\tilde{\Psi}_{\posetB}"] \\
      \Tor_{H^{*}(BL)}\bigl(\kk,\kk[\posetB]\bigr) \arrow{rr}{\AA^{*}} & & \Tor_{H^{*}(BL)}\bigl(\kk,\kk[\posetB]\bigr)
    \end{tikzcd}
  \end{equation*}
  commute. Here the domain of~\(\AA\) is equipped with the twisted \(*\)-product and the codomain with the canonical product.
\end{proposition}

\begin{proof}
  We apply \Cref{thm:main} to the case where \(\phi\) is the identity map on~\(\parq{\posetB}{\KK}\).
  We keep the canonical representative~\(c\in C_{1}(S^{1})\) from \Cref{sec:nat-general} 
  and choose \(\cA=\tilde{c}\) as in~\eqref{eq:def-tildec}.
  Moreover, we take \(b_{v}=-\smallhalf[e_{v},-e_{v}]\in C_{2}(T)\) for~\(v\in V=V'\), so that
  \begin{equation}
    d\,b_{v} = - \smallhalf\,[e_{v}] - \smallhalf\,[-e_{v}] = \tilde{c}_{v}-c_{v}.
  \end{equation}

  To compute \(\hatq_{ij}\), we first evaluate \eqref{eq:def-bij} and~\eqref{eq:def-wvj}
  to get the following data:
  \begin{equation}
    s=1,
    \qquad
    \lambda_{v,1}=-\smallhalf,
    \qquad
    \wwonevsv{\vA}{1}{\vB} = -\wwtwovsv{\vA}{1}{\vB} =
    \begin{cases}
      1 & \text{if \(\vA=\vB\),} \\
      0 & \text{otherwise}
    \end{cases}
  \end{equation}
  for \(\vA, \vB\in V\). Applying \Cref{thm:hatq}, we obtain the twisting terms 
  \begin{equation}
    \hatq_{ij}=\smallhalf\,q_{ij} \quad \text{for \(1\leq j<i\leq \nB\)}.
  \end{equation}
  
  Because \(\phi\) is the identity map, so is the map~\(\Ktw\).
  The map~\(\AA\) therefore is of the claimed form and in particular bijective.
  It induces an algebra map in cohomology because of the commutativity of the diagram above,
  which is the translation of the diagram from \Cref{thm:main} to our context.
\end{proof}

\begin{remark}
  Since we have used the isomorphisms~\(\Psi_{\posetB}\) and~\(\tilde\Psi_{\posetB}\)
  from \Cref{thm:iso-mult-2,thm:iso-mult-general}, our argument that
  \( 
    \AA^{*}\colon H^{*}(\Kl_{\posetB}) \to H^{*}(\Kl_{\posetB})
  \) 
  is multiplicative is topological in nature.
  It would be instructive to see a purely algebraic proof of this fact.
\end{remark}

We conclude with our recurring example.

\begin{example}
  \label{ex:nontrivial-qij-invert2}
  Consider once again the toric variety~\(\XXX_\Sigma\) from \Cref{ex:nontrivial-qij-mult}.
  By \Cref{thm:invert2} and formula~\eqref{eq:twisting terms in example} we have
  \begin{gather}
    \AA^{*}(a_{1}) = a_{1},
    \qquad
    \AA^{*}(a_{2}) = a_{2},
    \qquad
    \AA^{*}(c) = c, \\
    \AA^{*}(b) = b + \smallhalf\bigl[q_{21}+q_{32}-q_{31}\bigr] = b+c, 
  \end{gather}
  hence \(\AA\) does not induce the identity map in cohomology. Moreover,
  \begin{align}
    \AA^{*}(a_{1} * a_{2}) &= \AA^{*}(b-c) = b = \bigl[\alpha_{1}\,\alpha_{2} + \alpha_{2}\,\alpha_{3} - \alpha_{1}\,\alpha_{3}\bigr] \\
    \notag &= \bigl[(\alpha_{1}-\alpha_{3}) (\alpha_{2}-\alpha_{3})\bigr] = a_{1}\,a_{2} = \AA^{*}(a_{1})\,\AA^{*}(a_{2}),
  \end{align}
  which illustrates that \(\AA^{*}\) is multiplicative, in contrast to the identity map.
 \end{example}

\end{document}